\documentclass[twoside,12pt]{article} 

\usepackage[utf8]{inputenc}
\usepackage{amssymb,amsmath,amsthm,subfiles,mathtools,mathrsfs,comment,dsfont,mathrsfs,graphicx,enumerate,caption,wrapfig,float,hyperref,titlesec}

\makeatletter
\newcommand*{\rom}[1]{\expandafter\@slowromancap\romannumeral #1@}
\makeatother

\usepackage{titleps,lipsum}% http://ctan.org/pkg/{titleps,lipsum}

\usepackage{caption}
\usepackage{subcaption}

\usepackage{afterpage}

\usepackage[T1]{fontenc}

\author{\scshape Lian Haeming \\ \textit{Queen Mary University of London}}

\graphicspath{{\string~/Desktop/Tex Files/Lyapunov rectangle boxes.jpg}}
    
\hypersetup{
    colorlinks=true,
    linkcolor=blue,
    citecolor = blue,
    linktoc = all
}

\newcommand{\bigslant}[2]{{\left.\raisebox{.1em}{$#1$}\middle/\raisebox{-.1em}{$#2$}\right.}}
    
\DeclareMathOperator{\arcosh}{arcosh}

\usepackage{xcolor}
\usepackage[margin=0.72in]{geometry}
\numberwithin{equation}{section}
\title{ {\scshape \bfseries \Large  on the bandwidths of periodic approximations to discrete schr\"odinger operators } } 
\date{}
\newcommand\restr[2]{{% we make the whole thing an ordinary symbol
  \left.\kern-\nulldelimiterspace % automatically resize the bar with \right
  #1 % the function
  \vphantom{\big|} % pretend it's a little taller at normal size
  \right|_{#2} % this is the delimiter
  }}
  
\newcommand\PR{\mathds{P}}
\newcommand\E{\mathds{E}}
\newcommand\R{\mathbb{R}}
\newcommand\C{\mathbb{C}}
\newcommand\Z{\mathbb{Z}}
\newcommand\N{\mathbb{N}}
\newcommand\Q{\mathbb{Q}}
\newcommand\1{\mathds{1}}

\makeatletter
\newcommand{\newreptheorem}[2]{\newtheorem*{rep@#1}{\rep@title}\newenvironment{rep#1}[1]{\def\rep@title{#2 \ref*{##1}}\begin{rep@#1}}{\end{rep@#1}}}
\makeatother

\def\restrict#1{\raise-.5ex\hbox{\ensuremath|}_{#1}}
\newtheorem{thm}{\normalfont\scshape  Theorem}
\newtheorem{prop}[thm]{\normalfont\scshape Proposition}
\newtheorem{cor}[thm]{\normalfont\scshape Corollary}

\newtheorem{fig}{\normalfont\scshape Fig}

\newtheorem{lemma}{\normalfont\scshape  Lemma}[section]
\newtheorem{claim}[lemma]{\normalfont\scshape Claim}
\newtheorem{rmk}[lemma]{\normalfont\scshape Remark}

\newreptheorem{theorem}{\normalfont\scshape Theorem}
\newreptheorem{lemma}{\normalfont\scshape Lemma}
\newreptheorem{prop}{\normalfont\scshape Proposition}
\newreptheorem{cor}{\normalfont\scshape Corollary}

\makeatletter
\let\hdrtitle\@title% To store the title after using \maketitle
\newpagestyle{main}{%
  \sethead[][\scshape{Bandwidths of Periodic Approximations}][\thepage] % even
  {}{\scshape{L. Haeming}}{\thepage}} % odd
\makeatother
\titlespacing*{\section}{0pt}{4ex}{1.5ex}

\titleformat{\section}[block]{\color{black}\scshape\filcenter}{\thesection.}{0.5em}{}

\begin{document}
\renewcommand{\abstractname}{\vspace{-\baselineskip}}
\maketitle

\begin{abstract}
We study how the spectral properties of ergodic Schr\"odinger operators are reflected in the asymptotic properties of its periodic approximation as the period tends to infinity. The first property we address is the asymptotics of the bandwidths on the logarithmic scale, which quantifies the sensitivity of the finite volume restriction to the boundary conditions. We show that the bandwidths can always be bounded from below in terms of the Lyapunov exponent. Under an additional assumption satisfied by i.i.d potentials, we also prove a matching upper bound. Finally, we provide an additional assumption which is also satisfied in the i.i.d case, under which the corresponding eigenvectors are exponentially localised with a localisation centre independent of the Floquet number. 
\end{abstract}

%{
%  \hypersetup{linkcolor=black}
%  \tableofcontents
%}

\section*{Introduction}
Let $(\Omega,\mathcal{F},\PR,T)$ be an ergodic dynamical system and for convenience assume that the transformation $T$ is invertible. Let $f:\Omega\rightarrow \R$ be a bounded measurable function. To each outcome $\omega\in \Omega$ we associate a bounded discrete one-dimensional Schr\"odinger operator 
\begin{equation}\label{ergodicop}
H_{\omega}:\ell^2(\mathbb{Z}) \rightarrow \ell^2(\mathbb{Z}),\quad H_{\omega} = \Delta+V_{\omega}
\end{equation}
where $(\Delta\psi)(k) = \psi(k-1)+\psi(k+1) $ is the one-dimensional discrete free Laplacian and  $V_{\omega}$ is a multiplication operator defined by  $(V_{\omega}\psi)(k) = V_{\omega}(k)\psi(k) = f(T^{k}\omega)\psi(k)$. 

The main examples of ergodic Schr\"odinger operators \eqref{ergodicop} are: (I). Independent identically distributed (i.i.d) potentials, which correspond to the dynamical system 
\begin{equation*}
(\Omega,\mathcal{F},\PR,T) = \left(\Omega_{0}^{\Z},\mathcal{B}\left(\Omega_{0}^{\Z}\right),\mu^{\Z},T\right)
\end{equation*}
 where  $\Omega_{0}\subset\R$ is a bounded Borel set, $T$ is the left shift, $\mu$ is a Borel measure on $\Omega_{0}$ and   $f(\omega)=\omega(0)$. (II). One-frequency quasi-periodic potentials, which correspond to the irrational shift 
 \begin{equation*}
 (\Omega,\mathcal{F},\PR,T) = \big(\bigslant{\R}{\Z},\mathcal{B}\big(\bigslant{\R}{\mathbb{Z}}\big), m, \omega\mapsto\omega+\alpha\big)
 \end{equation*} where $m$ is the one-dimensional Lebesgue measure and $\alpha\in\R\setminus\Q$. 

We study the restriction of the ergodic operator \eqref{ergodicop} to a large finite interval $[0,q-1]$. The first property we investigate is the sensitivity of these restrictions to the boundary conditions. 

For convenience we shall represent the finite volume restriction of the operator \eqref{ergodicop} as an operator acting on functions on the discrete circle $\Z_{q} = \Z/q\Z$ with the representatives $\overline{0}_{q},\dots,\overline{q-1}_{q}$, as opposed to the usual equivalent notion of a matrix acting on the space $\C^{[0,q-1]}$. To this end, for $\varkappa\in\R/\frac{2\pi}{q}\Z$,  define
\begin{equation}\label{floquetmatrix}
\begin{split}
H_{q}[\omega,\varkappa]&:\C^{\Z_{q}}\rightarrow \C^{\Z_{q}} \\ 
(H_{q}[\omega,\varkappa]\psi)(x) = e^{i\varkappa}\psi(x-\overline{1}_{q}&) + e^{-i\varkappa}\psi(x+\overline{1}_{q}) + V_{\omega,q}(x)\psi(x),
\end{split}
\end{equation}
where $V_{\omega,q}(\overline{k}_{q})=V_{\omega}(k)$ for $k\in[0,q-1]$. The Floquet numbers $\varkappa=0,\frac{\pi}{q}$ correspond to periodic and anti-periodic boundary conditions, respectively. The difference between an eigenvalue of $H_{q}[\omega,0]$ and that of $H_{q}[\omega,\frac{\pi}{q}]$ quantifies the sensitivity to boundary conditions.

\begin{comment}
In this paper we study the restriction of the ergodic operator \eqref{ergodicop} to a large finite interval $[0,q-1]$. The first property we investigate is the sensitivity of these restrictions to the boundary conditions. For $\varkappa\in\R/\frac{2\pi}{q}\Z$, we denote 
\begin{equation}\label{floquetmatrix}
A_{\omega,q}(\varkappa) = 
\begin{pmatrix}
           V_{\omega}(0) & e^{-i\varkappa}         &          &           &     e^{i\varkappa} \\
           e^{i\varkappa}    & V_{\omega}(1)& e^{-i\varkappa}        &           &      \\
                &  e^{i\varkappa}        & \ddots   & \ddots   &       \\
                &           &  \ddots  &\ddots      &    e^{-i\varkappa}  \\
            e^{-i\varkappa}   &           &          &    e^{i\varkappa}       &  V_{\omega}(q-1)
\end{pmatrix}.
\end{equation}
$\varkappa=0,\frac{\pi}{q}$ correspond to periodic and anti-periodic boundary conditions, respectively. The difference between an eigenvalue of $A_{\omega,q}(0)$ and that of $A_{\omega,q}(\frac{\pi}{q})$ quantifies the sensitivity to boundary conditions. 
\end{comment}

In the context of random Schr\"odinger operators in arbitrary dimensions, Edwards and Thouless \cite{EdTh} put forth the sensitivity to periodic and anti-periodic boundary conditions as a signature of Anderson localisation: in the regime of localisation, this quantity should decay exponentially in the period. In the subsequent works \cite{Thouless1}-\cite{Thouless4}, Thouless focused on the critical almost Mathieu operator given by (II) with $f(\omega) = 2\cos(2\pi\omega)$, part of the conclusions of the analysis of \cite{Thouless1}-\cite{Thouless4} was justified on the mathematical level of rigour in the work of Last \cite{ZeroMeasureSpec}, see further the paper of Jitomirskaya,  Konstantinov and Krasovsky \cite{LanaKonstantinovKrasovsky}. 

The operators $H_{q}[\omega,\varkappa]$ also appear in the spectral analysis of the periodic approximations to $H_{\omega}$. Let $\widetilde{V}_{\omega,q}(k)=f(T^{k\text{ mod } q} \omega)$ and let $H_{\omega,q}=\Delta+\widetilde{V}_{\omega,q}$ be the periodic approximation to $H_{\omega}$. By Floquet theory (see e.g. \cite{Last-Periodic} Proposition 2.1 and Section \ref{sec:1}, below) the operator $H_{\omega,q}$ is unitarily equivalent to the direct integral of $H_{q}[\omega,\varkappa]$ over $\varkappa\in\R/\frac{2\pi}{q}\Z$. In particular, the spectrum of $H_{\omega,q}$ is the union over the eigenvalues of $H_{q}[\omega,\varkappa]$ and $\varkappa\in \R/\frac{2\pi}{q}\Z$. In the ergodic case \eqref{ergodicop} let $B_{\omega,q}^{(j)} $ be the $j$-th band of the periodic operator $H_{\omega,q}$ ordered from left to right and let $b_{\omega,q}^{(j)}$ be its centre. 

In the quasi-periodic case (II) there is an alternative notion of periodic approximation: let $\alpha_{n}=\frac{p_{n}}{q_{n}}\rightarrow \alpha\in\R\setminus\Q$ and $V_{\omega}^{(\alpha_{n})}(k) = f(\omega +k \alpha_{n})$. We let $B_{\omega}^{(\alpha_{n},j)}$ be the $j$-th band of  $H_{\omega}^{(\alpha_{n})} = \Delta + V_{\omega}^{(\alpha_{n})}$ ordered from left to right and let $b_{\omega}^{(\alpha_{n},j)}$ be its centre.

Here we study the asymptotics of the bandwidths on the logarithmic scale. It turns out that in many cases 
\begin{equation}\label{approx}
-q^{-1}\log m(B_{\omega,q}^{(j)}) \approx \gamma(b_{\omega,q}^{(j)})
\end{equation}
where $\gamma(E) = \lim_{q\rightarrow\infty} q^{-1} \E\log\|\Phi_{\omega,q}(E)\|$ is the Lyapunov exponent. $\Phi_{\omega,q}(E)$ is the $q$-step transfer matrix of \eqref{ergodicop}, given by $\Phi_{\omega,q}(E)=
\begin{pmatrix}
E-V_{\omega}(q-1) & -1 \\
1 & 0
\end{pmatrix} 
\cdots 
\begin{pmatrix}
E-V_{\omega}(0) & -1 \\
1 & 0
\end{pmatrix}
$.

\begin{figure}[H]
\centering
\begin{subfigure}{.5\textwidth}
  \centering
  \includegraphics[width=1\linewidth, height=0.3\textheight]{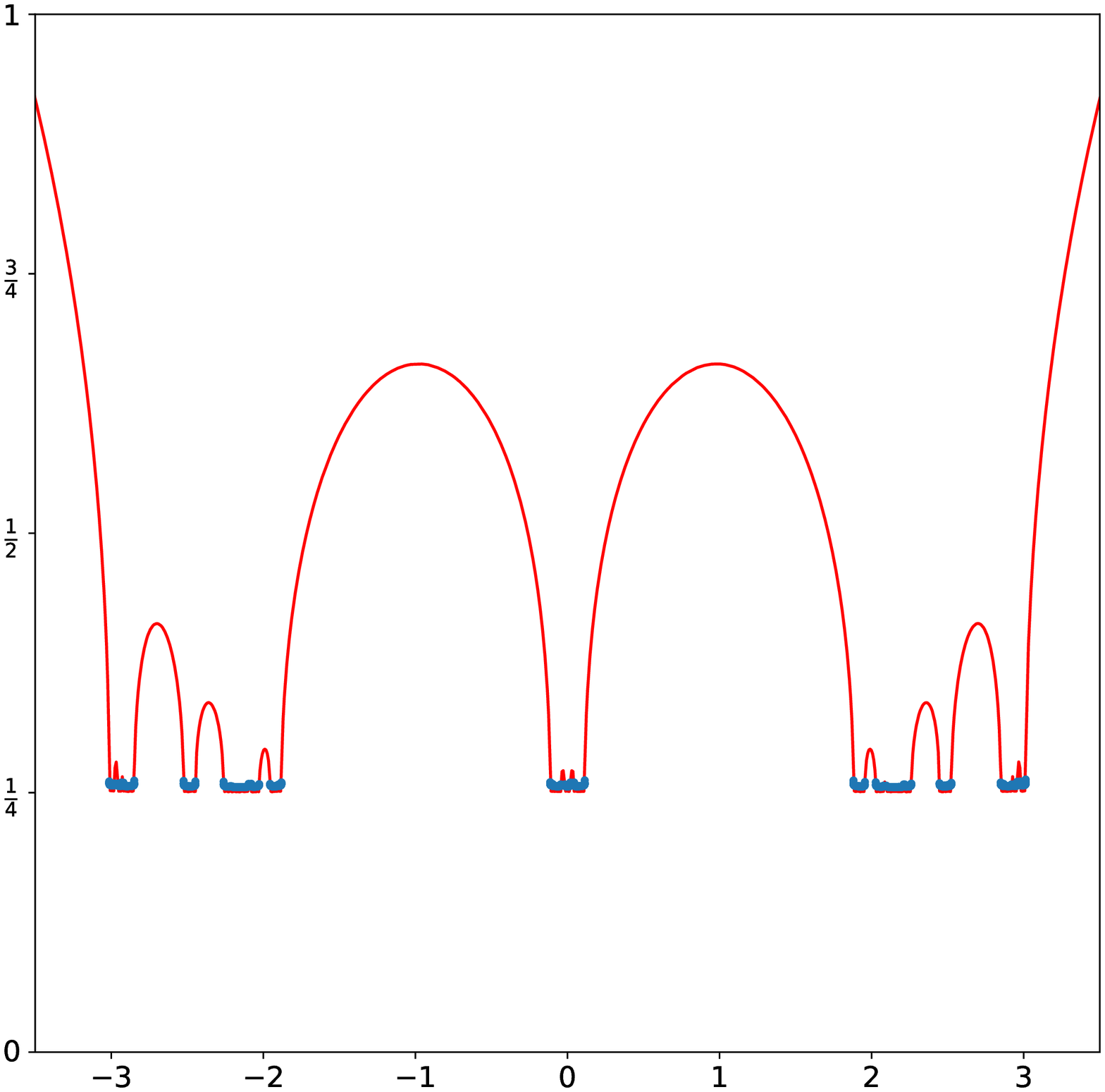} 
\end{subfigure}%
\begin{subfigure}{.5\textwidth}
  \centering
  \includegraphics[width=1\linewidth, height=0.3\textheight]{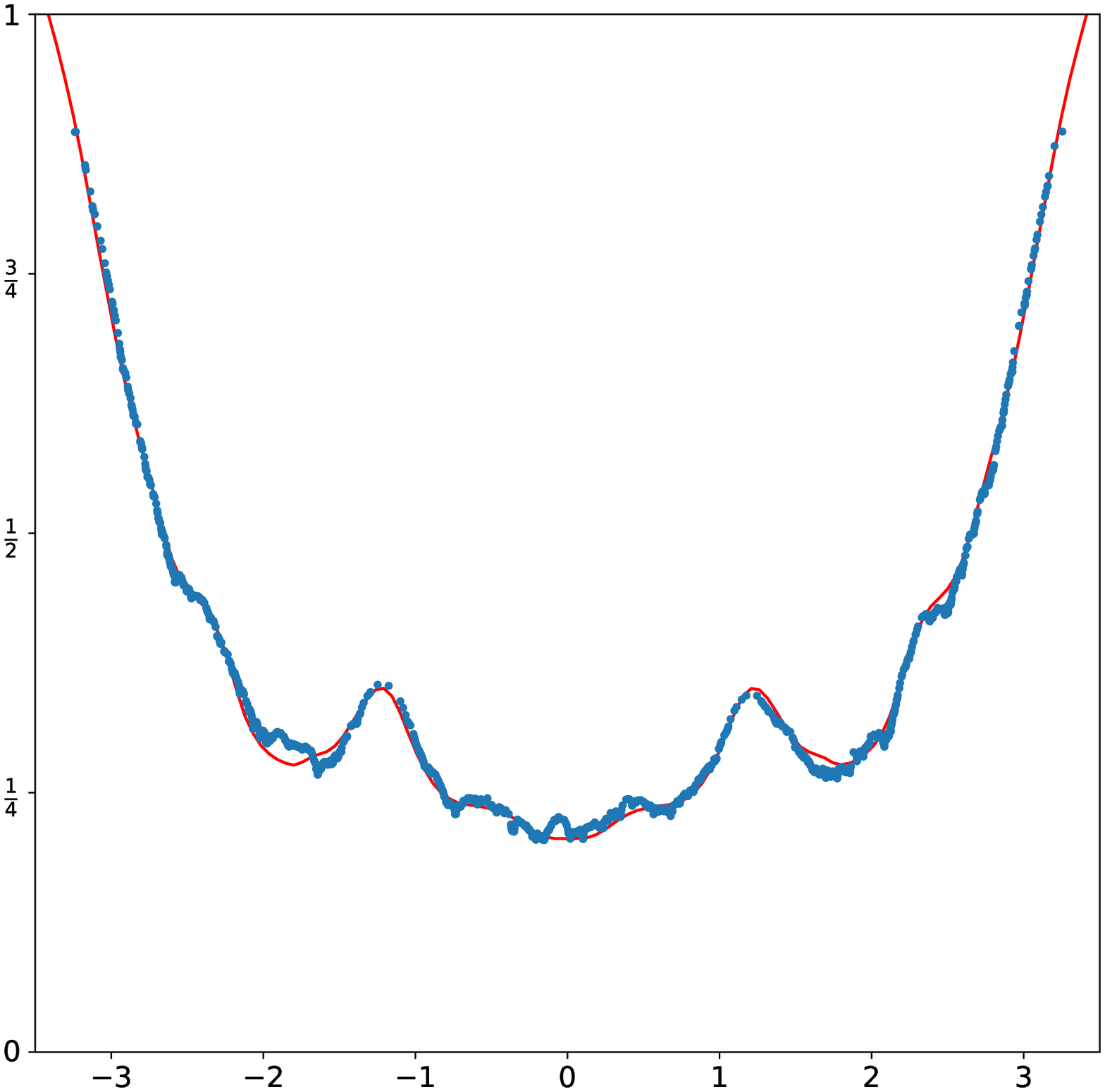} 
\end{subfigure}
\end{figure}
\begin{fig}
The plot on the left corresponds to the almost Mathieu operator with potential function given by $ 2e^{\frac{1}{4}}\cos(2\pi (\sqrt{3}+k\sqrt{2}))$ whereas the right one corresponds to an i.i.d operator with underlying uniform distribution on $[-\frac{3}{2},-1]\cup[1,\frac{3}{2}]$. The blue dots on the left and on the right have coordinates given by $\big(b_{\sqrt{3}}^{(\frac{1393}{985},\cdot)}, -985^{-1}\log  m(B_{\sqrt{3}}^{(\frac{1393}{985},\cdot)})\big)$ and $\big(b_{\omega,1400}^{(\cdot)},-1400^{-1}\log m(B_{\omega,1400}^{(\cdot)})\big)$, respectively. In both cases, the red line represents the Lyapunov exponent of the respective operator. 
\end{fig}

Our Theorem \ref{ergodiccontinuous} justifies one direction of \eqref{approx} while Theorem \ref{cl:9}
justifies the second direction under an additional assumption. As we discuss below, this assumption holds with large probability in the i.i.d case.

\begin{thm}\label{ergodiccontinuous}
If $H_{\omega}$ is an ergodic Schr\"odinger operator \eqref{ergodicop} with continuous Lyapunov exponent $\gamma$, then with full probability 
$$
\limsup_{q\rightarrow \infty}\max_{j\in[1,q]} \big( -q^{-1}\log m(B_{\omega,q}^{(j)})-\gamma(b_{\omega,q}^{(j)}) \big) \leq 0.
$$
If $T$ is a uniquely ergodic transformation then   
$$\limsup_{q\rightarrow \infty}\max_{j\in[1,q],\omega\in\Omega} \big( -q^{-1}\log m(B_{\omega,q}^{(j)})-\gamma(b_{\omega,q}^{(j)}) \big) \leq 0.$$
 In the quasi-periodic case \emph{(II)}, if $\alpha_{n}=\frac{p_{n}}{q_{n}}\rightarrow \alpha\in\R\setminus\Q$ is a sequence of rationals then, 
$$
\limsup_{n\rightarrow \infty}\max_{j\in[1,q_{n}], \omega\in\Omega}\big(-q_{n}^{-1}\log m(B_{\omega}^{(\alpha_{n},j)})-\gamma(b_{\omega}^{(\alpha_{n},j)}) \big) \leq 0.$$
\end{thm}

\begin{rmk}
 Bourgain and  Jitomirskaya show in \cite{ContinuousLyapunov} that the Lyapunov exponent is continuous in the case of quasi-periodic operators (II) with analytic $f$. S. Klein \cite{SilviusKlein} extended the results of \cite{ContinuousLyapunov} to the case that $f$ satisfies a Gevrey condition:   $\sup|f^{(k)}|\leq AB^{k}(k!)^{s}$ for some $1\leq s<2$ and is not flat at any point.
\end{rmk}

\begin{rmk}
Continuity and in fact H\"older continuity of the Lyapunov exponent for i.i.d potentials (I) goes back to the work of  Le Page \cite{LePage2}. 
\end{rmk}

In order to state the assumption on the potential for Theorem \ref{cl:9}, we first define the set of resonant sites on the circle. Let $H_{q}[\omega,\varkappa]\upharpoonright \Lambda$ denote  the restriction of the operator $H_{q}[\omega,\varkappa]$ to the arc $\Lambda\subseteq\Z_{q}$, with Dirichlet boundary conditions.  We define the Green function of the operator $H_{q}[\omega,\varkappa]$ as 
\begin{equation}\label{greens}
G_{\omega,\varkappa,\Lambda,E,q}= (H_{q}[\omega,\varkappa]\upharpoonright \Lambda-E)^{-1}
\end{equation}
for $E\notin \sigma(H_{q}[\omega,\varkappa]\upharpoonright \Lambda)$. For $x,y\in\Z_{q}$ we write $G_{\omega,\varkappa,\Lambda,E,q}(x,y) = \langle G_{\omega,\varkappa,\Lambda,E,q} \delta_{y},\delta_{x}\rangle$ where $\delta_{x}$ is a delta function at $x\in\Z_{q}$.
 
Let $B_{n}(x)$ denote the arc of half-length $0\leq n\leq \lfloor\frac{q}{2}\rfloor$ centred at $x\in\Z_{q}$. We say that a site $x\in\Z_{q}$ is $(E,\varepsilon,n)$-non-resonant if the Green function \eqref{greens} satisfies 
$$
|G_{\omega,0,B_{n}(x),E,q}(x,x\pm \overline{n}_{q})| <  e^{-(\gamma(E)-\varepsilon)n}
$$
and if this is the case we write $x\in \text{Res}^{\mathsf{c}}_{\omega}(E,\varepsilon,n) = \Z_{q}\setminus \text{Res}_{\omega}(E,\varepsilon,n)$. 
 
The assumption on the potential is that for every energy $E\in K$ where $K$ is a closed interval containing the spectrum, the diameter of the set of resonant sites $\text{Res}_{\omega}(E,\varepsilon,n)$ is at most $2n$ and if this is the case then we write $\omega\in  Q_{\text{NR}}(\varepsilon,n,q)$. 

\begin{thm}\label{cl:9}
Let $H_{\omega}$ be an ergodic Schr\"odinger operator \eqref{ergodicop} with continuous Lyapunov exponent $\gamma$. For any $\varepsilon>0$ there exists $q_{0} = q_{0}(\varepsilon)$ such that if $q>q_{0}$ and $e^{-\frac{1}{2}n\min_{\R}\gamma}<\frac{1}{2}$ then for any $\omega\in Q_{\emph{NR}}(\varepsilon,n,q)$, 
$$
\min_{j\in[1,q]}\big(-q^{-1}\log m\big(B_{\omega,q}^{(j)}\big) -\gamma\big(b_{\omega,q}^{(j)}\big) \big)\geq -3\varepsilon.
$$
\end{thm}

This result substantiates in this setting the idea put forth by Edwards and Thouless in the 1970s.

The next proposition shows that in the i.i.d.\ case, Theorem \ref{cl:9} holds with a probability which is exponentially close to $1$ in the period. 

\begin{prop}\label{c0}
Let $H_{\omega}$ be a Schr\"odinger operator with bounded i.i.d potential \emph{(I)}. For any $\varepsilon>0$ there exists $N_{1}=N_{1}(\varepsilon)$ and $c=c(\varepsilon)$ such that if $n>N_{1}$, then
$\PR( Q_{\emph{NR}}(\varepsilon,n,q)) > 1-q^{2}e^{-cn}$.
\end{prop}

\begin{rmk}
Statements of this kind appear in the classical works on localisation in the infinite volume case, for example in \cite{CKM} using multi-scale analysis. Here we adapt it to the finite volume case for periodic operators following the more recent strategy of  Jitomirskaya and Zhu \cite{LanaZhu} of obtaining a large deviation estimate on the Green function \eqref{greens}. 
\end{rmk}

Applying the Borel-Cantelli lemma to Theorem \ref{cl:9} and combining this with the first part of Theorem \ref{ergodiccontinuous} gives, in the i.i.d.\ case, the exact bandwidth distribution in the limit: 

\begin{cor}
Let $H_{\omega}$ be a Schr\"odinger operator with bounded i.i.d potential \emph{(I)}. With full probability  
$$\lim_{q\rightarrow\infty} \max_{j\in[1,q]}\left|-q^{-1}\log m\big(B_{\omega,q}^{(j)}\big)-\gamma\big(b_{\omega,q}^{(j)}\big)\right| = 0.$$
\end{cor}

Our next result pertains to the Anderson localisation of the eigenvectors of the operator $H_{q}[\omega,\varkappa]$. 

The arguments from the theory of Anderson localisation yield that the eigenvectors of $H_{q}[\omega,\varkappa]$ are exponentially localised for any fixed Floquet number $\varkappa\in\R/\frac{2\pi}{q}\Z$. The formal statement (and proof) appears as Lemma \ref{distancedecay} below and is the main ingredient used in the proof of Theorem \ref{cl:9}. 

It is natural to ask whether the localisation centres actually depend on the Floquet number. We show in the next theorem that this is not the case, under the additional assumption on the potential of eigenvalue separation: Let $Q_{\text{Sep}}(\varepsilon,q)$ denote the set of outcomes $\omega\in\Omega$ for which  the eigenvalues of the operator $H_{q}[\omega,0]$ are separated by at least $e^{-\varepsilon q}$. 

\begin{thm}\label{uniformlemma}
Let $H_{\omega}$ be an ergodic Schr\"odinger operator \eqref{ergodicop} and $H_{q}[\omega,\varkappa]\psi_{\omega,\varkappa}^{(j)}=E_{\omega,\varkappa}^{(j)}\psi_{\omega,\varkappa}^{(j)}$. For any $0<4\varepsilon<\min_{\R}\gamma$ there exists $C=C(\varepsilon)$ and $q_{0}=q_{0}(\varepsilon)$ such that if $q>q_{0}$ and $e^{-\frac{1}{2}n\min_{\R}\gamma }<\frac{1}{2}$ then for any  $\omega\in Q_{\emph{Sep}}(\frac{\varepsilon}{10},q)\cap Q_{\emph{NR}}(\varepsilon,n,q)$ there exists $\{\nu_{\omega,j}\}_{j}\subseteq\Z_{q}$ such that for each $j=1,\dots,q$, if $\|x-\nu_{\omega,j}\|_{q}>Cn$ then 
\begin{equation*}
\max_{\varkappa\in\R/\frac{2\pi}{q}\Z}\big| \psi_{\omega,\varkappa}^{(j)}(x)\big| < e^{-\big(\gamma\big( E_{\omega,0}^{(j)}\big)-3\varepsilon\big)\|x-\nu_{\omega,j}\|_{q}}
 \end{equation*} where $ \|x-y\|_{q}=\min_{x_{0}\in x,y_{0}\in y}|x_{0}-y_{0}|$. 
\end{thm}

A lower bound on the spacings between eigenvalues was established by Bourgain \cite{BourgainSpacings} for the case of the Bernoulli shift with Dirichlet boundary conditions. We build on the methods in \cite{BourgainSpacings} to show the following for bounded i.i.d potentials (I) with periodic boundary conditions. 

\begin{prop}\label{eigensep}
Let $H_{\omega}$ be a Schr\"odinger operator with bounded i.i.d potential \emph{(I)}. There exists $C,c,\varepsilon_{0}$ such that if  $Cq^{-1}<\varepsilon<\varepsilon_{0}$ then $\PR( Q_{\emph{Sep}}(\varepsilon,q))>1-q^{2}e^{-c\varepsilon q}$.
\end{prop}

\begin{rmk}
The argument in \cite{BourgainSpacings} and the one presented here require the use of transfer matrices and are  therefore limited to the one-dimensional case. If the cumulative distribution function corresponding to the distribution $\mu$ is uniformly H\"older of order $\frac{1}{2}+\delta$ then a conclusion similar to that of Proposition \ref{eigensep}, for arbitrary dimensions, also follows from the Minami estimate \cite{Minami}. 
\end{rmk}

\begin{rmk}
We rely on  the positivity of the Lyapunov exponent. In the i.i.d case (I),  Furstenberg’s theorem \cite{Furstenberg} implies that $\gamma>0$. 
\end{rmk}

\emph{Structure}. We start the first section with a quick reminder of Floquet theory in order to establish the lower bounds on the bandwidths. Afterwards, in the second section, we touch on the Anderson localisation of the eigenvectors of the operator \eqref{floquetmatrix}, in which we permit the localisation centres to depend on the Floquet number and show that this is enough to obtain the upper bound on the bandwidths. In the third section we show that given the additional assumption of eigenvalue separation, the localisation centres do not depend on the Floquet number. In the remaining sections we will show that the assumptions in Theorems \ref{cl:9} and \ref{uniformlemma} do indeed hold in the i.i.d.\ setting, with large probability. 
 
From here onwards we will work on $K=[-\sup_{(\omega,q)}\|H_{\omega,q}\|-10 ,\sup_{(\omega,q)}\|H_{\omega,q}\| +10]$ so almost surely the spectrum $\sigma(H_{\omega,q})$ does not intersect with $K^{\mathsf{c}}$. We drop the dependence of any quantities on $K$ and   denote $\min\gamma = \min_{\R}\gamma$.

\section*{Acknowledgements}
The author is grateful to his advisor Mira Shamis for proposing this work and to Sasha Sodin for his extensive support towards its completion. This work was supported by the Engineering and Physical Science Research Council.

\section{Lower bounds.}\label{sec:1}
In this section we prove Theorem \ref{ergodiccontinuous}. Our argument requires a uniform (in the energy) upper bound on the gradient of the characteristic polynomial corresponding to the matrix representation of the operator \eqref{floquetmatrix}. Let us first recall a few facts from Floquet theory. The operator \eqref{floquetmatrix} is represented in the standard basis, by the so-called Floquet matrix: 
\begin{equation}
A_{\omega,q}(\varkappa) = 
\begin{pmatrix}
           V_{\omega,q}(\overline{0}_{q}) & e^{-i\varkappa}         &          &           &     e^{i\varkappa} \\
           e^{i\varkappa}    & V_{\omega,q}(\overline{1}_{q})& e^{-i\varkappa}        &           &      \\
                &  e^{i\varkappa}        & \ddots   & \ddots   &       \\
                &           &  \ddots  &\ddots      &    e^{-i\varkappa}  \\
            e^{-i\varkappa}   &           &          &    e^{i\varkappa}       &  V_{\omega,q}(\overline{q-1}_{q})
\end{pmatrix}.
\end{equation}
The conjugation $D_{\varkappa}A_{\omega,q}(\varkappa)D_{\varkappa}^{-1}$ by  $D_{\varkappa}= \text{Diag}(1,\dots,e^{-i\varkappa (q-1)})$, allows us to compute the dependence on $\varkappa\in\R/\frac{2\pi}{q}\Z$, of its characteristic polynomial, which is given by 
\begin{equation}\label{charpoly}
    \det(A_{\omega,q}(\varkappa)-E) = \Delta_{\omega,q}(E)+2(-1)^{q-1}\cos(q\varkappa)
\end{equation}
where the discriminant $\Delta_{\omega,q}(E)$ is a polynomial of degree $q$ with real coefficients. 

The set $\Delta_{\omega,q}^{-1}([-2,2])$ is made up of $q$ closed intervals $B_{\omega,q}^{(j)}=\cup_{\varkappa}\big\{E_{\omega,\varkappa}^{(j)}\big\}$ which may intersect, but only at the edges.  The spectrum of the operator $H_{\omega,q}$ is given by the union of the roots of the characteristic polynomial \eqref{charpoly} over $\varkappa\in \R/\frac{2\pi}{q}\Z$, so too we have $\sigma(H_{\omega,q})=\Delta_{\omega,q}^{-1}([-2,2])$. Moreover, 
\begin{equation}\label{bandintegral}
4 = \int_{B_{\omega,q}^{(j)}}  |\Delta'_{\omega,q}(E)| \,dE
\leq m(B_{\omega,q}^{(j)}) \max_{E\in B_{\omega,q}^{(j)}} |\Delta'_{\omega,q}(E)|
\end{equation}
whence Theorem \ref{ergodiccontinuous} follows provided we bound the gradient of the discriminant uniformly on the bands, thereby motivating Lemma \ref{thm1}, below. 

The following well-known lemma (see e.g.\ Deift-Simon \cite{DeiftSimon}) allows us to deduce that for a large enough period, the Lyapunov exponent only varies very little on the bands due to our assumption of continuity of the Lyapunov exponent. 

\begin{lemma}\label{smallbands}
The largest bandwidth of discrete $q$-periodic Schr\"odinger operators is at most $ \frac{2\pi}{q}$.
\end{lemma}
\begin{proof}
Let $H_{q}[\varkappa]$ be defined as in \eqref{floquetmatrix}.  Let $H_{q}[\varkappa]\psi_{\varkappa}=E_{\varkappa}\psi_{\varkappa}$, $B=\cup_{\varkappa}E_{\varkappa}$ and $\|\psi_{\varkappa}\|=1$. Since $H_{q}[\varkappa]$ is self-adjoint, Feynman-Hellmann and Cauchy-Schwarz implies 
\begin{equation}\label{eq:201}
\left|\frac{dE_{\varkappa}}{d\varkappa}\right| = \left|\left\langle \frac{dH_{q}[\varkappa]}{d\varkappa} \psi_{\varkappa},\psi_{\varkappa} \right\rangle \right| \leq \sum_{x\in\Z_{q}} (|\psi_{\varkappa}(x-\overline{1}_{q})|+|\psi_{\varkappa}(x+\overline{1}_{q})|)|\psi_{\varkappa}(x)|\leq 2
\end{equation}
and by monotonicity of the derivative $\frac{dE_{\varkappa}}{d\varkappa}$ the characteristic polynomial \eqref{charpoly} of the Floquet matrix implies 
$m(B) =  \int_{0}^{\frac{\pi}{q}} \left|\frac{dE_{\varkappa}}{d\varkappa}\right| \,d\varkappa\leq \frac{2\pi}{q}$. 
\end{proof}

\begin{rmk} The constant $2\pi$ in Lemma \ref{smallbands} is actually optimal. The middle band in the spectrum of the  discrete free Laplacian, seen as a periodic operator of period $q$, has bandwidth $\sim\frac{2\pi}{q}\,(q\rightarrow\infty)$. Shamis-Sodin \cite[Theorem 1.2]{ShamisSodin} shows that the $j$-th band from the left of a periodic operator is at  most as big as the $j$-th band from the left of the discrete free Laplacian. 
\end{rmk}

In the general ergodic setting: let $\Delta_{\omega,q}$ denote the discriminant corresponding to the periodic operator $H_{\omega,q}$. In the quasi-periodic setting (II): let  $\alpha_{n}=\frac{p_{n}}{q_{n}}\rightarrow\alpha\in\R\setminus\Q$  be a sequence of rationals and let $\Delta_{\omega,q,\alpha_{n}}$ denote the discriminant corresponding to the $q$-periodic operator $H^{(\alpha_{n})}_{\omega,q} = \Delta+V_{\omega,q}^{(\alpha_{n})}$ where   $V_{\omega,q}^{(\alpha_{n})}(x) = f(T_{\alpha_{n}}^{x\text{ mod } q} \omega)$ and $T_{\alpha_{n}}\omega = \omega + \alpha_{n}$.

\begin{lemma}\label{thm1}
If $H_{\omega}$ is an ergodic Schr\"odinger operator \eqref{ergodicop} with continuous Lyapunov exponent $\gamma$, then with full probability 
$$
\limsup_{q\rightarrow \infty} \sup_{E\in K} (q^{-1}\log |\Delta_{\omega,q}'(E)| - \gamma(E) ) \leq 0. 
$$
If $T$ is a uniquely ergodic transformation then   
$$
\limsup_{q\rightarrow \infty}\sup_{(\omega,E)\in  \Omega\times K} (q^{-1}\log |\Delta_{\omega,q}'(E)| - \gamma(E) ) \leq 0.
$$
 In the quasi-periodic case \emph{(II)}, if $\alpha_{n}=\frac{p_{n}}{q_{n}}\rightarrow \alpha\in\R\setminus\Q$ is a sequence of rationals then, 
$$
\limsup_{n\rightarrow\infty} \sup_{(\omega,E)\in  \Omega\times K}(q_{n}^{-1}\log|\Delta'_{\omega,q_{n},\alpha_{n}}(E)|-\gamma(E)) \leq 0.
$$
\end{lemma}

We first complete the proof of Theorem \ref{ergodiccontinuous}. 

\begin{proof}[Proof of Theorem \ref{ergodiccontinuous}]
General ergodic case: Fix $\varepsilon>0$ and $\omega\in\Omega$ in the event of full probability in the first statement of Lemma \ref{thm1}. There exists a constant $q_{1}=q_{1}(\varepsilon,\omega)$ such that if $q>q_{1}$, then for each $j = 1,\dots,q$ there exists $\varkappa_{j}\in\R/\frac{2\pi}{q}\Z$ such that 
$$\max_{E\in B_{\omega,q}^{(j)}} |\Delta'_{\omega,q}(E)| \leq \max_{E\in B_{\omega,q}^{(j)}} e^{(\gamma(E)+\varepsilon)q} = e^{\big(\gamma\big(E_{\omega,\varkappa_{j}}^{(j)}\big)+\varepsilon\big)q}.$$  

By Lemma \ref{smallbands} we have $\max_{j} m\big(B_{\omega,q}^{(j)}\big) \leq \frac{2\pi}{q}$, so for any $\delta>0$ there exists $q_{2}=q_{2}(\delta)$ such that if $q>q_{2}$, then $\max_{j}m\big(B_{\omega,q}^{(j)}\big)<\delta$. Since the Lyapunov exponent is uniformly continuous on $K$, there exists  $\delta = \delta(\varepsilon)>0$ for which $\big|\gamma(E)-\gamma(E')\big|< \varepsilon$ for any $E,E'\in K$ with $|E-E'|<\delta$.  For $q>q_{2}$, $\big|\gamma\big(E_{\omega,\varkappa_{j}}^{(j)}\big)-\gamma\big(b_{\omega,q}^{(j)}\big)\big|\leq2\varepsilon$ and  
$e^{\big(\gamma\big(E_{\omega,\varkappa_{j}}^{(j)}\big)+\varepsilon\big)q} \leq e^{\big(\gamma\big(b_{\omega,q}^{(j)}\big)+3\varepsilon\big)q}$
for each $j=1,\dots,q$.

Let $q^{-1}\log 4  \leq 2\varepsilon$ for $q>q_{3}=q_{3}(\varepsilon)$, then \eqref{bandintegral} implies
\begin{equation}\label{eq:132}
m\big(B_{\omega,q}^{(j)}\big) \geq e^{-\big(\gamma\big(b_{\omega,q}^{(j)}\big)+\varepsilon\big)q} 
\end{equation}
for every $j=1,\dots,q$ and $q>q_{0}(\varepsilon,\omega)=\max\{q_{1}(\varepsilon,\omega),q_{2}(\delta(\varepsilon)),q_{3}(\varepsilon)\}$. 

Uniquely ergodic case: The same reasoning as above. Uniformity follows from Lemma \ref{thm1}.

Quasi-periodic case: By the third statement of Lemma \ref{thm1} there exists $N_{0}=N_{0}(\varepsilon)$ such that for $n>N_{0}$,
\begin{equation}\label{eq:124}
|\Delta'_{\omega,q_{n},\alpha_{n}}(E)| \leq e^{(\gamma(E) + \varepsilon)q_{n}}
\end{equation}
for every $(\omega,E)\in \Omega\times K $. By the same reasoning that led to \eqref{eq:132}, we can increase $N_{0}(\varepsilon)$ such that an analogous lower bound to \eqref{eq:132} for the present case, follows for $n> N_{0}$. Uniformity in $\omega\in\Omega$ follows from the fact that \eqref{eq:124} holds uniformly in $\omega$.
\end{proof}

We now prove the first and the third statements of Lemma \ref{thm1}. We will omit the proof of the second statement as it is almost identical to that of the third one.

\begin{proof}[Proof of the first statement of Lemma \ref{thm1}]
Let  $\Phi_{\omega,q}$ denote the $q$-step transfer matrix of the operator  $H_{\omega}$.  A theorem of Craig-Simon \cite{C-S} states that $\PR$-almost surely for every energy $E\in\R$ we have $$\limsup_{q\rightarrow\infty} q^{-1}\log \|\Phi_{\omega,q}(E)\| \leq \gamma(E)$$ in the case of general ergodic Schr\"odinger operators \eqref{ergodicop}. The discriminant satisfies $\Delta_{\omega,q} = \text{Tr}(\Phi_{\omega,q})$ and so 
$|\text{Tr}(\Phi_{\omega,q})| \leq 2\|\Phi_{\omega,q}\|$ implies that $\PR$-almost surely for every $E\in\R$, 
\begin{equation}\label{eq:130}
\limsup_{q\rightarrow\infty} q^{-1}\log |\Delta_{\omega,q}(E)| \leq \gamma(E).
\end{equation}

We upgrade \eqref{eq:130} to an estimate that holds locally uniformly in the energy by using the Remez inequality together with Egoroff's theorem. 

Fix $\varepsilon>0$ and take $\omega\in \Omega$ from the event of full probability which satisfies \eqref{eq:130}. We begin by showing that for any $  E_{0}\in K$, there exists $q_{1}=q_{1}(\varepsilon,\omega,  E_{0})$ and $\delta=\delta(\varepsilon)>0$ such that $
q^{-1}\log |\Delta_{\omega,q}(E)| - \gamma(E) \leq \varepsilon $
for all $E\in B_{\delta}(  E_{0})$ and $q>q_{1}$. 

Since the Lyapunov exponent $\gamma$ is uniformly continuous on $K$, let $\delta = \delta(\varepsilon)>0$ be such that $\gamma(E')-\varepsilon\leq \gamma(E)$ for every $E\in B_{\delta}(E')$ and every $E'\in K$. Fix $  E_{0}\in K$, by Egoroff's theorem there exists a subset $B\subset B_{\delta}(  E_{0})$ with  $m(B_{\delta}(  E_{0})\setminus B)=\widetilde{\delta}$ where $\widetilde{\delta}<\frac{2\delta\varepsilon^{2}}{1+\varepsilon^{2}}$, on which the limit \eqref{eq:130} is uniform on $E\in B$. 

We show
\begin{equation}\label{eq:82}
\sup_{E\in B_{\delta}(  E_{0})}|\Delta_{\omega,q}(E)| \leq  e^{2q\left(\frac{\widetilde\delta}{m(B)}\right)^{\frac{1}{2}}} \sup_{E\in B}|\Delta_{\omega,q}(E)|.
\end{equation}
Indeed, the Remez inequality states 
\begin{equation}\label{eq:81}
\sup_{E\in B_{\delta}(  E_{0})}|\Delta_{\omega,q}(E)| \leq  T_{q}(1+\rho) \sup_{E\in  B}|\Delta_{\omega,q}(E)|
\end{equation}
where $\rho=\frac{2\widetilde{\delta}}{m(B)}>0$ and $T_{q}$ is the first kind Chebyshev polynomial of degree $q$ which satisfies $T_{q}(1+\rho)=\cosh(q\arcosh(1+\rho))$.

We claim that $T_{q}(1+\rho)\leq e^{q\sqrt{2\rho}}$, which together with \eqref{eq:81} proves \eqref{eq:82}. Indeed, let $t>0$ be such that $\arcosh(1+\rho)=t$, then expanding the hyperbolic cosine into its Taylor series we get $1+\rho = \cosh(t)>1+\frac{t^{2}}{2}$ which implies $0<t<\sqrt{2\rho}$, $\arcosh(1+\rho)<\sqrt{2\rho}$ and $T_{q}(1+\rho)<\cosh(q\sqrt{2\rho}) < e^{q\sqrt{2\rho}}$.

By uniformity of \eqref{eq:130} in the energy on $B$, there exists $q_{1}=q_{1}(\varepsilon,\omega,  E_{0})$  such that 
$q^{-1}\log|\Delta_{\omega,q}(E)| - \gamma(E) \leq \varepsilon$
uniformly in $E\in B$ if $q>q_{1}$. Uniform continuity of the Lyapunov exponent implies $\sup_{E\in B} q^{-1}\log|\Delta_{\omega,q}(E)| \leq \gamma(  E_{0})+2\varepsilon$
whenever $q>q_{1}$. The Remez inequality \eqref{eq:81}, our choice of $\delta$ and the fact that $\widetilde{\delta} = \widetilde{\delta}(\varepsilon)<\frac{2\delta\varepsilon^{2}}{1+\varepsilon^{2}}$, imply
\begin{equation}\label{eq:90}
\sup_{E\in B_{\delta}(  E_{0})}(q^{-1}\log |\Delta_{\omega,q}(E)|-\gamma(E)) 
\leq \frac{2\sqrt{\widetilde\delta}}{\sqrt{2\delta-\widetilde{\delta}}}+3\varepsilon \leq 5\varepsilon
\end{equation} for $q>q_{1}$.

Fix  $  E_{0}\in K$, let  $\delta=\delta(\varepsilon)$ and $q_{1}=q_{1}(\varepsilon,\omega,  E_{0})$ be as above. We claim that there exists $q_{3}=q_{3}(\varepsilon,\omega,  E_{0})$ such that 
$q^{-1}\log |\Delta_{\omega,q}'(E)| \leq \gamma(E) + \varepsilon$
for all $E\in B_{\delta}(E_{0})$ if $q>q_{3}$. Indeed, the Markov inequality gives 
\begin{equation}\label{eq:78}
\sup_{E\in B_{\delta}(  E_{0})}(q^{-1}\log |\Delta_{\omega,q}'(E)|-\gamma(E))
\leq q^{-1}\log \left( \frac{2q^{2}}{m(B_{\delta}(  E_{0}))} \right) + \sup_{E\in B_{\delta}(  E_{0})} (q^{-1}\log|\Delta_{\omega,q}(E)|-\gamma(E))
\end{equation}
and there exists $q_{2}=q_{2}(\varepsilon)$ such that $q^{-1}\log\frac{2q^{2}}{2\delta}<\frac{\varepsilon}{2}$ for all $q>q_{2}$, therefore \eqref{eq:78} implies that $\sup_{E\in B_{\delta}(  E_{0})}(q^{-1}\log |\Delta_{\omega,q}'(E)|-\gamma(E)) \leq \varepsilon$ for $q>q_{3}=q_{3}(\varepsilon,\omega,  E_{0}) = \max\{q_{1}(\frac{\varepsilon}{2},\omega,  E_{0}),q_{2}(\varepsilon)\}$.

The set $\{B_{\delta}(E)\}_{E\in K}$ is an open cover for the compact interval $K$ so there exists a finite subset $X\subset K$ such that $K\subset \{B_{\delta}(E)\}_{E\in X}$. Let $q_{0}=q_{0}(\varepsilon,\omega)=\max\{q_{3}(\varepsilon,\omega,E):E\in X\}$, then for any $\varepsilon>0$, we have $\sup_{q>q_{0}} \sup_{E\in K} (q^{-1}\log |\Delta_{\omega,q}'(E)| - \gamma(E) ) \leq \varepsilon$.
\end{proof}

\begin{proof}[Proof of the third statement of Lemma \ref{thm1}]
Let $\Phi_{\omega,q}$ denote the $q$-step transfer matrix of the quasi-periodic operator $H_{\omega}$ defined over the irrational shift by $\alpha\in\R\setminus\Q$. Let $\Phi_{\omega,q,\alpha_{n}}$ denote the $q$-step transfer matrix of the operator $H_{\omega,q}^{(\alpha_{n})}$ approximating $H_{\omega,q}$.  A theorem of Furman \cite{Furman} states that for any energy $E\in\R$, 
\begin{equation}\label{eq:204}
\limsup_{q\rightarrow\infty} q^{-1}\log \|\Phi_{\omega,q}(E)\| \leq \gamma(E)
\end{equation}
holds uniformly in $\omega\in \Omega$ in the case where $T$ is a uniquely ergodic transformation. The irrational shift is uniquely ergodic so the theorem indeed applies to $\Phi_{\omega,q}$. 

We begin by showing that for any $\varepsilon>0$, there exists $q_{2}=q_{2}(\varepsilon)$ and $N_{0}=N_{0}(\varepsilon)$ such that if $q>q_{2}$ and $n>N_{0}$ then 
\begin{equation}\label{uniformity1}
q^{-1}\log\|\Phi_{\omega,q,\alpha_{n}}(E)\|\leq \gamma(E) + \varepsilon 
\end{equation}
for every $(\omega,E)\in \Omega\times K $.

Denote $ f_{q,\alpha_{n}}(\omega,E)=\log\|\Phi_{\omega,q,\alpha_{n}}(E)\|$ and $f_{q}(\omega,E)=\log\|\Phi_{\omega,q}(E)\|$. Fix $E_0\in K$ and $\varepsilon>0$, by \eqref{eq:204} there exists $\widetilde{q} = \widetilde{q}(  E_{0},\varepsilon)$ such that $\widetilde{q}^{-1}f_{\widetilde{q}}(\omega,  E_{0}) \leq \gamma(  E_{0})+\frac{\varepsilon}{8}$ for all $\omega\in\Omega$. The function $f_{\widetilde{q}}$ is continuous in $(\omega,E)\in \Omega\times K$ and thus uniformly continuous on $ \Omega\times K$ so there exists $\delta=\delta(\varepsilon)>0$ such that $|f_{\widetilde{q}}(\omega, E)-f_{\widetilde{q}}(\omega,   E_{0})|<\frac{\varepsilon}{8} $ for all $(\omega,E)\in \Omega\times B_{\delta}(E_0)$. Since the entries of the $q$-step transfer matrix $\Phi_{\omega,q,\alpha_{n}}$ converge to the corresponding entries of $\Phi_{\omega,q}$, uniformly in $(\omega,E)$,  let $N_{0}=N_{0}(\varepsilon)$ be such that $\max_{(\omega,E)\in \Omega\times K}|f_{\widetilde{q}}(\omega,E) - f_{\widetilde{q},\alpha_{n}}(\omega,E)|<\frac{\varepsilon}{4}$, whenever $n>N_{0}$. Then 
\begin{equation}\label{fixed_{q}}
{\widetilde{q}}^{-1}f_{\widetilde{q},\alpha_{n}}(\omega,E)
\leq \widetilde{q}^{-1}f_{\widetilde{q}}(\omega,E) + \frac{\varepsilon}{8} 
\leq \widetilde{q}^{-1}f_{\widetilde{q}}(\omega,E_0) + \frac{2\varepsilon}{8} 
\leq \gamma(E_0) + \frac{3\varepsilon}{8}
\end{equation}
for all $(\omega,E)\in \Omega \times B_{\delta}(E_0)$, whenever $n>N_{0}$ where we have assumed $\widetilde{q}\geq 2$ in the first inequality. 

 By subadditivity and uniformity in $\omega\in\Omega$ of \eqref{fixed_{q}}, we have
$$
f_{k\widetilde{q}+r,\alpha_{n}}(\omega,E) 
\leq f_{r,\alpha_{n}}(T_{\alpha_{n}}^{k\widetilde{q}}\omega,E) +\sum_{j=0}^{k-1}f_{\widetilde{q},\alpha_{n}}(T_{\alpha_{n}}^{j\widetilde{q}}\omega,E)
\leq r \max_{(\omega,E)\in \Omega \times K} |f_{1,\alpha_{n}}(\omega,E)| + k\widetilde{q}\left(\gamma(E_0) + \frac{3\varepsilon}{8}\right)
$$
there exists $C_{0}$ such that $ \max_{(\omega,E)\in \Omega \times K} |f_{1,\alpha_{n}}(\omega,E)|\leq C_{0}$ for any $n$, thus 
$$
\frac{1}{k\widetilde{q}+r}f_{k\widetilde{q}+r,\alpha_{n}}(\omega,E) \leq 
\frac{r}{k\widetilde{q}+r}C_{0} + \frac{k\widetilde{q}}{k\widetilde{q}+r}\left(\gamma(E_0)+\frac{3\varepsilon}{8}\right).
$$
Let $k_{0}=k_{0}(\varepsilon)$ be such that $k^{-1}C_{0} < \frac{\varepsilon}{8}$ for $k>k_{0}$, then $(k\widetilde{q}+r)^{-1}f_{k\widetilde{q}+r,\alpha_{n}}(\omega,E)\leq \gamma(E_0)+\frac{\varepsilon}{2}$ for all $(\omega,E)\in \Omega \times B_{\delta}(E_0)$, for any $k>k_{0}$ and $r=0,\dots,\widetilde{q}-1$. Let $q_{1}=q_{1}(\varepsilon,  E_{0}) = k_{0}\widetilde{q}+1$, then 
$q^{-1} f_{q,\alpha_{n}}(\omega,E)\leq \gamma(  E_{0})+\frac{\varepsilon}{2}$
 for all $(\omega,E)\in\Omega\times B_{\delta}(  E_{0})$ whenever $q>q_{1}$ and $n>N_{0}$.

By uniform continuity of the Lyapunov exponent $\gamma$ on $K$, there exists $\delta' = \delta'(\varepsilon)$ such that $|\gamma(E)-\gamma(E')|<\frac{\varepsilon}{2}$ whenever $|E-E'|<\delta'$. Let $\widetilde{\delta} = \widetilde{\delta}(\varepsilon) = \min\{\delta(\varepsilon),\delta'(\varepsilon)\}$. Then for any $\varepsilon>0$ and $  E_{0}\in K$, we have $q^{-1} f_{q,\alpha_{n}}(\omega,E)\leq \gamma(  E_{0}) +\frac{\varepsilon}{2} \leq \gamma(E)+\varepsilon$ for any $(\omega,E)\in\Omega\times B_{\widetilde{\delta}}(  E_{0})$ whenever $q>q_{1}$ and $n>N_{0}$.

$\{B_{\widetilde{\delta}}(E)\}_{E\in K}$ is an open cover for  $K$. Let $X\subset K$ be a finite subset such that $K\subset \{B_{\widetilde{\delta}}(E)\}_{E\in X}$. Define $q_{2}(\varepsilon)=\max_{E\in X}q_{1}(\varepsilon,E)$,  if $q>q_{2}$ and $n>N_{0}$ then \eqref{uniformity1} follows for all $(\omega,E)\in \Omega\times K$. 

Now that we have uniformity of the norm, we differentiate the trace to obtain the same for the derivative of the discriminant. Let $P$ be the projection $P(x,y)^\intercal=(x,0)^\intercal$, then 
$$\Delta'_{\omega,q,\alpha_{n}}(E)=\frac{d}{dE}\text{Tr}(\Phi_{\omega,q,\alpha_{n}}(E))=\sum_{j=0}^{q-1}\text{Tr}(\Phi_{T_{\alpha_{n}}^{q-j}\omega,j,\alpha_{n}}(E)P\Phi_{\omega,q-(j+1),\alpha_{n}}(E))$$ 
where we take $\Phi_{\cdot,0,\cdot}=I$. Then $|\Delta_{\omega,q,\alpha_{n}} | \leq 2\|\Phi_{\omega,q,\alpha_{n}}\|$ implies 
\begin{equation}\label{eq:125}
|\Delta'_{\omega,q,\alpha_{n}}(E)|\leq 2q\max_{0\leq j\leq q-1} \|\Phi_{T_{\alpha_{n}}^{q-j}\omega,j,\alpha_{n}}(E)\|\|\Phi_{\omega,q-(j+1),\alpha_{n}}(E)\|.
\end{equation}
Let $S_{\omega,q,\alpha_{n}}(E)$ denote the maximum on the RHS of \eqref{eq:125}. For each $n,q\geq 1$ there exists  $\zeta_{n}(q) = \zeta(q)$ with $ 0 \leq \zeta(q)\leq q-1$, such that 
$S_{\omega,q,\alpha_{n}}(E)=\|\Phi_{T_{\alpha_{n}}^{q-\zeta(q)}\omega,\zeta(q),\alpha_{n}}(E)\|\|\Phi_{\omega,q-(\zeta(q)+1),\alpha_{n}}(E)\|$. \eqref{uniformity1} implies $\|\Phi_{\omega,q,\alpha_{n}}(E)\|\leq e^{(\gamma(E)+\varepsilon)q}$ for all $(\omega,E)\in \Omega\times K$, whenever $q>q_{2}$ and $n>N_{0}$. 

Consider $$q^{-1}\log S_{\omega,q,\alpha_{n}}(E)=q^{-1}\frac{\zeta(q)}{\zeta(q)}\log\|\Phi_{T_{\alpha_{n}}^{q-\zeta(q)}\omega,\zeta(q),\alpha_{n}}(E)\| + q^{-1}\frac{q-(\zeta(q)+1)}{q-(\zeta(q)+1)}\log\|\Phi_{\omega,q-(\zeta(q)+1),\alpha_{n}}(E)\|.$$ 
For $q_{3}=q_{3}(\varepsilon)$ to be determined later, we have 
\begin{equation*}
\begin{split} 
\{q\geq q_{3}\} 
&= \{q\geq q_{3}: \zeta(q),q-(\zeta(q)+1)\leq q_{2}\} \cup\{q\geq q_{3}: \zeta(q)\leq q_{2}<q-(\zeta(q)+1)\}  \\
&\cup \{q\geq q_{3}: q-(\zeta(q)+1)\leq q_{2} < \zeta(q)\}\cup \{q\geq q_{3}: q_{2}<\zeta(q),q-(\zeta(q)+1)\}\\
&= X_{0}\cup X_{1}\cup X_{2}\cup X_{3}.
\end{split}
\end{equation*}

If $q\in X_{3}$, then \eqref{uniformity1} implies $q^{-1}\log S_{\omega,q,\alpha_{n}}(E) \leq q^{-1}(\zeta(q)+q-(\zeta(q)+1))(\gamma(E)+\varepsilon) \leq \gamma(E)+\varepsilon$ for any $(\omega,E)\in \Omega\times K$. If $q\in X_{2}$, then $q^{-1}\log S_{\omega,q,\alpha_{n}}(E) \leq \gamma(E)+\varepsilon+q^{-1}\log\|\Phi_{\omega,q-(\zeta(q)+1),\alpha_{n}}(E)\|\leq \gamma(E)+\varepsilon+q^{-1}C_{1}$ and there exists $q_{4}=q_{4}(\varepsilon)$ such that $q^{-1}C_{1}<\varepsilon$ whenever $q>q_{4}$. In particular, 
\begin{equation}\label{eq:75}
q^{-1}\log S_{\omega,q,\alpha_{n}}(E) \leq \gamma(E)+2\varepsilon
\end{equation}
 for all $(\omega,E)\in \Omega\times K$ whenever $q\in X_{2}$ where $q_{3}>q_{4}$. 

A similar argument can be made for $q\in X_{1}$, for which there exists $q_{5}$ such that $q^{-1}\widetilde{C}_{1}<\varepsilon$ whenever $q>q_{5}$ which  implies \eqref{eq:75} for all $(\omega,E)\in \Omega\times K$ when $q\in X_{1}$, with $q_{3}>q_{5}$. Note that if $q_{3}\geq 2(q_{2}+1)$, then $X_{0}=\varnothing$. Take $q_{3} > \max \{2(q_{2}+1),q_{4},q_{5}\}$, then we have \eqref{eq:75} for all $(\omega,E)\in \Omega\times K$ whenever $q>q_{3}$ and $n>N_{0}$. 

The dependence of $\zeta$ on $n$ is not problematic since \eqref{eq:75} holds for any $n>N_{0}$. Take $q_{6} = q_{6}(\varepsilon) \geq q_{3}$ such that $q>q_{6}$ implies $q^{-1}\log(2q)\leq \varepsilon$, then take  $q_{0}(\varepsilon) = q_{6}(\frac{\varepsilon}{3})$. Increase $N_{0}(\varepsilon)$ such that $q_{n}>q_{0}(\varepsilon)$ whenever $n>N_{0}$, then for $\varepsilon>0$, 
$\sup_{n>N_{0}} \sup_{(\omega,E)\in  \Omega\times K}(q_{n}^{-1}\log|\Delta'_{\omega,q_{n},\alpha_{n}}(E)|-\gamma(E)) \leq \varepsilon$.
\end{proof}

\section{Upper bound.}\label{sec:2}

We use a standard subharmonicity trick to show that the assumptions $\omega\in Q_{\text{NR}}(\varepsilon,n,q)$ and $e^{-\frac{1}{2}n\min\gamma }<\frac{1}{2}$ are enough to obtain localisation of the eigenvectors $\psi_{\omega,\varkappa}^{(j)}$ for  fixed $\varkappa\in\R/\frac{2\pi}{q}\Z$,  allowing the localisation centre to vary with the Floquet number and from this we show that  Theorem \ref{cl:9} follows.

\begin{lemma}\label{distancedecay}
Let $H_{\omega}$ be an ergodic Schr\"odinger operator \eqref{ergodicop}. If $e^{-\frac{1}{2}n\min\gamma }<\frac{1}{2}$ then for any $\varepsilon>0$, there exists $C=C(\varepsilon)$  such that for any $\omega\in  Q_{\emph{NR}}(\varepsilon,n,q)$ and $\varkappa\in\R/\frac{2\pi}{q}\Z$ there exists a set $ \{\nu_{\omega,j,\varkappa}\}_{j}\subseteq\Z_{q}$ of localisation centres such that for each $j=1,\dots,q$, if $\|x-\nu_{\omega,j,\varkappa}\|_{q}>Cn$, then $\big| \psi_{\omega,\varkappa}^{(j)}(x)\big|\leq e^{-\big(\gamma\big( E_{\omega,\varkappa}^{(j)}\big)-2\varepsilon\big)\|x-\nu_{\omega,j,\varkappa}\|_{q}}.$
\end{lemma}
\begin{proof}
Fix $\varepsilon>0$, $\varkappa\in\R/\frac{2\pi}{q}\Z$, $e^{-\frac{1}{2}n\min\gamma }<\frac{1}{2}$ and $\omega\in  Q_{\text{NR}}(\varepsilon,n,q)$. One can show that if $x\in\text{Res}^{\mathsf{c}}_{\omega}(E_{\omega,\varkappa}^{(j)},\varepsilon,n)$ and $H_{q}[\omega,\varkappa]\psi_{\omega,\varkappa}^{(j)}=E_{\omega,\varkappa}^{(j)}\psi_{\omega,\varkappa}^{(j)}$, then 
\begin{equation*}
\psi_{\omega,\varkappa}^{(j)} (x)=-G_{\omega,\varkappa,B_{n}(x), E_{\omega,\varkappa}^{(j)},q}(x,x-\overline{n}_{q}) \psi_{\omega,\varkappa}^{(j)}(x-\overline{n+1}_{q})-G_{\omega,\varkappa,B_{n}(x), E_{\omega,\varkappa}^{(j)},q}(x,x+\overline{n}_{q}) \psi_{\omega,\varkappa}^{(j)} (x+\overline{n+1}_{q}).
\end{equation*}
Moreover, we have $|G_{\omega,\varkappa,B_{n}(x), E_{\omega,\varkappa}^{(j)},q}|=|G_{\omega,0,B_{n}(x), E_{\omega,\varkappa}^{(j)},q}|$ since, as in the beginning of Section \ref{sec:1}, the two operators are equal up to a conjugation by $\psi(\overline{k}_{q})\mapsto e^{ik\varkappa}\psi(\overline{k}_{q})$ for $\overline{k}_{q}\in B_{n}(x)$, $k\in[0,q-1]$, thus $x\in\text{Res}^{\mathsf{c}}_{\omega}\big( E_{\omega,\varkappa}^{(j)},\varepsilon,n\big)$ implies 
 \begin{equation}\label{eq:166}
\big|\psi_{\omega,\varkappa}^{(j)} (x)\big|\leq e^{-\big(\gamma\big( E_{\omega,\varkappa}^{(j)}\big)-\varepsilon\big)n}\big(\big|\psi_{\omega,\varkappa}^{(j)} \big(x+\overline{n+1}_{q}\big)\big|+\big|\psi_{\omega,\varkappa}^{(j)} \big(x-\overline{n+1}_{q}\big)\big|\big).
\end{equation} 

It's enough to consider only those $j$ for which  $2\varepsilon<\gamma\big(E_{\omega,\varkappa}^{(j)}\big)$ since we assume $\big\|\psi_{\omega,\varkappa}^{(j)}\big\|=1$. 
Let $\nu_{\omega,j,\varkappa}\in\Z_{q}$ be a maximum of the eigenvector $\psi_{\omega,\varkappa}^{(j)}$,  we show that $\nu_{\omega,j,\varkappa}\in\text{Res}_{\omega}\big( E_{\omega,\varkappa}^{(j)},\varepsilon,n\big)$. Indeed, since $e^{-\frac{1}{2}n\min\gamma }<\frac{1}{2}$, we deduce  from \eqref{eq:166} that for any non resonant $x$,
\begin{equation}\label{eq:250}
 \big|\psi_{\omega,\varkappa}^{(j)}(x)\big|<\max\big\{\big|\psi_{\omega,\varkappa}^{(j)}\big(x+\overline{n+1}_{q}\big)\big|,\big|\psi_{\omega,\varkappa}^{(j)}\big(x-\overline{n+1}_{q}\big)\big|\big\} = \big|\psi_{\omega,\varkappa}^{(j)}\big(x+\alpha\big(\overline{n+1}_{q}\big)\big)\big|
\end{equation}
  for some $\alpha\in\{1,-1\}$. We can iterate \eqref{eq:250} to get $\big|\psi_{\omega,\varkappa}^{(j)}\big(x+\alpha\big(\overline{n+1}_{q}\big)\big)\big| < \big|\psi_{\omega,\varkappa}^{(j)}\big(x+2\alpha\big(\overline{n+1}_{q}\big)\big)\big|$
and keep going until we find the smallest positive integer $k_{0}$ for which $x+k_{0}\alpha(\overline{n+1}_{q})\in \text{Res}_{\omega}^{\mathsf{c}}\big( E_{\omega,\varkappa}^{(j)},\varepsilon,n\big)$. This shows, in particular, that $\big|\psi_{\omega,\varkappa}^{(j)}(x)\big|<\big|\psi_{\omega,\varkappa}^{(j)}(x')\big|$ for some $x'\in\text{Res}_{\omega}^{\mathsf{c}}\big( E_{\omega,\varkappa}^{(j)},\varepsilon,n\big)$. If no such $k_{0}$ exists, then  B\'ezout's lemma leads to the contradiction $\big|\psi_{\omega,\varkappa}^{(j)}(x)\big|<\big|\psi_{\omega,\varkappa}^{(j)}(x)\big|$. Indeed, B\'ezout's lemma ensures the existence of $k_{1},k_{2}$ such that $(qk_{1})(n+1)+(qk_{2})q = q\,\text{gcd}(q,n+1)$. Since $x$ is an arbitrary non-resonant site, the same is true for all  non-resonant sites.

 Let $\Lambda_{j}=\Lambda_{j,\varkappa}\subset \mathbb{Z}_{q}$ be an arc containing the resonant sites corresponding to the $j$-th eigenvalue of $H_{q}[\omega,\varkappa]$ and with the same diameter.  For any $x=x^{(j,\varkappa)}\in\mathbb{Z}_{q}$ denote $d_{x}=d_{x}^{(j,\varkappa)}=\text{dist}\big(x, \text{Res}_{\omega}\big( E_{\omega,\varkappa}^{(j)},\varepsilon,n\big)\big)$. Since we are only concerned with $x\in\Z_{q}$ for which  $\|x-\nu_{\omega,j,\varkappa}\|_{q}>Cn$, we may assume that $C\geq 2$ and thus  consider only  $x\in  \Lambda_{j}^{\mathsf{c}}$. 
 
 Writing $d_{x} = (M+1)(n+1)-r$ for  some $1\leq r < n+1$, we have $x\pm m(\overline{n+1}_{q})\in \Lambda_{j}^{\mathsf{c}}$ for each $1\leq m\leq M$ otherwise $\text{dist}(x,\Lambda_{j}) \leq\|k(\overline{n+1}_{q})\|_{q} = m(n+1) < d_{x}$ since for each $1\leq m \leq M$ we have $m(n+1)\leq M(n+1)<\frac{q}{2}$. So we can iterate \eqref{eq:166} $ \widetilde{d}_{x}  =  \left\lceil \frac{d_{x}}{n+1}\right\rceil = M+1$ times, giving 
\begin{equation*}
\begin{split}
\big|\psi_{\omega,\varkappa}^{(j)}(x)\big|
&\leq e^{-\big(\gamma\big( E_{\omega,\varkappa}^{(j)}\big)-\varepsilon\big)n\widetilde{d}_{x}} \sum_{s=0}^{\widetilde{d}_{x}}{\widetilde{d}_{x}\choose s}\big|\psi_{\omega,\varkappa}^{(j)}\big(x+(2s-\widetilde{d}_{x})(\overline{n+1}_{q})\big)\big| \\
&\leq 2^{\widetilde{d}_{x}}e^{-\big(\gamma\big( E_{\omega,\varkappa}^{(j)}\big)-\varepsilon\big)n\widetilde{d}_{x}} \leq e^{-\big(\gamma\big( E_{\omega,\varkappa}^{(j)}\big)-\varepsilon\big)n\widetilde{d}_{x}+ \widetilde{d}_{x}\log3} \leq e^{-\big(\gamma\big( E_{\omega,\varkappa}^{(j)}\big)-\varepsilon\big)d_{x}}.
\end{split} 
\end{equation*}

We have $e^{-\big(\gamma\big(E_{\omega,\varkappa}^{(j)}\big)-\varepsilon\big)d_{x}}<e^{-\big(\gamma\big( E_{\omega,\varkappa}^{(j)}\big)-2\varepsilon\big)\|x-\nu_{\omega,j,\varkappa}\|_{q}}$ for appropriate $x\in\Z_{q}$. Indeed, write $r_{x}=r_{x}^{(j,\varkappa)}\leq 2n$ such that 
$d_{x}=\|x-\nu_{\omega,j,\varkappa}\|_{q}-r_{x}$. It suffices to show $\varepsilon > \frac{r_{x}\gamma\big( E_{\omega,\varkappa}^{(j)}\big)}{\|x-\nu_{\omega,j,\varkappa}\|_{q}+r_{x}},$ the RHS of which is at most $2C^{-1}\max_{E\in K}\gamma(E)$ since $0\leq r_{x}\leq 2n$ and $\|x-\nu_{\omega,j,\varkappa}\|_{q}>Cn$. It remains to take $C=C(\varepsilon)\geq 2$ satisfying   $2C^{-1}\max_{E\in K}\gamma(E)< \varepsilon$.
\end{proof}

\begin{proof}[Proof of Theorem \ref{cl:9}]
Fix $\varepsilon>0$, $e^{-\frac{1}{2}n\min\gamma}<\frac{1}{2}$ and $\omega\in Q_{\text{NR}}(\varepsilon,n,q)$. We show that for large enough period $q$ we have $m\big(B_{\omega,q}^{(j)}\big) \leq e^{-\big(\gamma\big(b_{\omega,q}^{(j)}\big)-3\varepsilon\big)q}$ for each $j=1,\dots,q$. By Lemma \ref{smallbands} we need only consider those $j$ for which $\gamma\big(b_{\omega,q}^{(j)}\big)>3\varepsilon$, assuming $q_{0}\geq 7$. 

By Lemma \ref{distancedecay} there exists $C=C(\varepsilon)$ such that for any $\varkappa\in\R/\frac{2\pi}{q}\Z$ there exists a set $\{\nu_{\omega,j,\varkappa}\}_{j}\subseteq\Z_{q}$ of localisation centres such that 
for each $j=1,\dots,q$, if $\|x-\nu_{\omega,j,\varkappa}\|_{q}>Cn$ then 
\begin{equation}\label{eq:200}
\big| \psi_{\omega,\varkappa}^{(j)}(x)\big| < e^{-\big(\gamma\big( E_{\omega,\varkappa}^{(j)}\big)-2\varepsilon\big)\|x-\nu_{\omega,j,\varkappa}\|_{q}}.
 \end{equation}

Define $(\mathcal{L}\psi)(x) = \psi(x+\overline{1}_{q})$ and $(D_{\varkappa}\psi)(\overline{k}_{q}) = e^{-ki\varkappa}$. For any $l\in\Z$ and $\varkappa\in\R/\frac{2\pi}{q}\Z$ define $H_{q}^{(l)}[\omega,\varkappa] = D_{\varkappa}\mathcal{L}^{l}H_{q}[\omega,\varkappa]\mathcal{L}^{-l}D_{\varkappa}^{-1}$ which in particular means that the eigenvalues of $H_{q}^{(l)}[\omega,\varkappa]$ do not depend on $l$. 

The dependence on $\varkappa$ of the operator $H_{q}^{(l)}[\omega,\varkappa]$ are exactly at $k=0$ and $k=q-1$, as seen by 
$$
(H_{q}^{(l)}[\omega,\varkappa]\psi)(\overline{k}_{q}) = e^{-ik\varkappa+i\varkappa(1+(k-1\text{ mod }q))}\psi(\overline{k-1}_{q})+V_{\omega,q}(\overline{k+l}_{q})\psi(\overline{k}_{q})+e^{-ik\varkappa+i\varkappa(-1+(k+1\text{ mod }q))}\psi(\overline{k+1}_{q})
$$ which has the same (up to a rotation of the diagonal) matrix representation as the conjugation given at the beginning of Section \ref{sec:1}. 

Let $\widetilde{\psi}_{\omega,\varkappa,l}^{(j)}  = D_{\varkappa}\mathcal{L}^{l}\psi_{\omega,\varkappa}^{(j)}$ then $H_{q}^{(l)}[\omega,\varkappa]\widetilde{\psi}_{\omega,\varkappa,l}^{(j)}=E_{\omega,\varkappa}^{(j)}\widetilde{\psi}_{\omega,\varkappa,l}^{(j)}$ and Feynman-Hellmann gives  
\begin{equation*}
\frac{dE_{\omega,\varkappa}^{(j)}}{d\varkappa} = \left\langle \frac{dH_{q}^{(l)}[\omega,\varkappa]}{d\varkappa} \widetilde{\psi}_{\omega,\varkappa,l}^{(j)}, \widetilde{\psi}_{\omega,\varkappa,l}^{(j)} \right\rangle =qie^{i\varkappa}\psi_{\omega,\varkappa}^{(j)}(\overline{l+q-1}_{q})\overline{\psi_{\omega,\varkappa}^{(j)}(\overline{l}_{q})}-
qie^{-i\varkappa}\psi_{\omega,\varkappa}^{(j)}(\overline{l}_{q})\overline{\psi_{\omega,\varkappa}^{(j)}(\overline{l+q-1}_{q})} 
\end{equation*}
which implies $\left|\frac{dE_{\omega,\varkappa}^{(j)}}{d\varkappa}\right| \leq 2q |\psi_{\omega,\varkappa}^{(j)}(\overline{l}_{q})||\psi_{\omega,\varkappa}^{(j)}(\overline{l+q-1}_{q})|$. Choose $l=l_{\omega,j,\varkappa}$ so that for each $j=1,\dots,q$ and $\varkappa\in\R/\frac{2\pi}{q}\Z$, \eqref{eq:200} gives 
\begin{equation}\label{eq:202}
\big|\psi_{\omega,\varkappa}^{(j)}\big(\overline{l}_{q}\big)\big|\big|\psi_{\omega,\varkappa}^{(j)}\big(\overline{l+q-1}_{q}\big)\big| \leq e^{-\big(\gamma\big( E_{\omega,\varkappa}^{(j)}\big)-2\varepsilon\big)(q-1)}.
\end{equation}

By Lemma \ref{smallbands}, for any $\delta>0$ there exists  $q_{2}=q_{2}(\delta)$ for which $\max_{j}m\big(B_{\omega,q}^{(j)}\big)<\delta$, whenever $q>q_{2}$.  Let  $\delta = \delta(\varepsilon)$ be such that $|\gamma(E)-\gamma(E')|< \frac{\varepsilon}{4}$ for any $E,E'\in K$ with $|E-E'|<\delta$. We have  $\big|\gamma\big( E_{\omega,\varkappa}^{(j)}\big)-\gamma\big(b_{\omega,q}^{(j)}\big)\big|<\frac{\varepsilon}{2}$ for any   $j=1,\dots,q$ and $\varkappa\in\R/\frac{2\pi}{q}\Z$ provided $q>q_{2}$. Applying \eqref{eq:202} to Feynman-Hellmann \eqref{eq:201} and $m\big(B_{\omega,q}^{(j)}\big)=\int_{0}^{\frac{\pi}{q}} \big|\frac{dE_{\omega,\varkappa}^{(j)}}{d\varkappa}\big| \,d\varkappa $, we obtain $m\big(B_{\omega,q}^{(j)}\big) \leq 2\pi e^{-\big(\gamma\big( b_{\omega,q}^{(j)}\big)-\frac{5}{2}\varepsilon\big)(q-1)}$ for each $j=1,\dots,q$. It follows that there exists $q_{3} = q_{3}(\varepsilon)$ such that if $q>q_{3}$ then the RHS is at most  $e^{-\big(\gamma\big(b_{\omega,q}^{(j)}\big)-3\varepsilon\big)q}$. It remains to take $q_{0}(\varepsilon) = \max\{7,q_{1},q_{2},q_{3}\}$. 
\end{proof}

\section{Uniform localisation.}\label{sec:3}

Now we show that the localisation centres of the eigenvectors of $H_{q}[\omega,\varkappa]$ are fixed for all $\varkappa\in\R/\frac{2\pi}{q}\Z$ and thus upgrade Lemma \ref{distancedecay} to Theorem \ref{uniformlemma}.

\begin{proof}[Proof of Theorem \ref{uniformlemma}]
Fix $0<4\varepsilon<\min\gamma$, $j\in[1,q]$, $e^{-\frac{1}{2}n\min\gamma }<\frac{1}{2}$ and take $\omega\in Q_{\text{Sep}}(\frac{\varepsilon}{10},q)\cap  Q_{\text{NR}}(\varepsilon,n,q)$. Let $C=C(\varepsilon)$ and $\{\nu_{\omega,j,\varkappa}\}_{j}\subseteq\Z_{q}$ be given by Lemma \ref{distancedecay}. Then by Lemma \ref{distancedecay}, $\|x-\nu_{\omega,j,0}\|_{q}>Cn$ implies 
$\big| \psi_{\omega,0}^{(j)}(x)\big| < e^{-\big(\gamma\big( E_{\omega,0}^{(j)}\big)-2\varepsilon\big)\|x-\nu_{\omega,j,0}\|_{q}}.$

Our goal is to show that for a large enough period $q$, 
 $\|x-\nu_{\omega,j,0}\|_{q}>Cn$ implies 
\begin{equation*}
\max_{\varkappa\in\R/\frac{2\pi}{q}\Z}\big| \psi_{\omega,\varkappa}^{(j)}(x)\big| < e^{-\big(\gamma\big( E_{\omega,0}^{(j)}\big)-3\varepsilon\big)\|x-\nu_{\omega,j,0}\|_{q}}.
 \end{equation*}

Fix $\varkappa\in\R/\frac{2\pi}{q}\Z$,  let $ \psi_{\omega,0}^{(j)} =\sum_{k}\alpha_{\omega,\varkappa}^{(k)} \psi_{\omega,\varkappa}^{(k)}$, write  $\varphi_{\omega,\varkappa}^{(j)}=\sum_{k\neq j}\alpha_{\omega,\varkappa}^{(k)} \psi_{\omega,\varkappa}^{(k)}$ and note $\varphi_{\omega,\varkappa}^{(j)}\perp \psi_{\omega,\varkappa}^{(j)}$. We show that $\big\|\varphi_{\omega,\varkappa}^{(j)}\big\|$ is small enough to ensure that the perturbed eigenvector $\psi_{\omega,\varkappa}^{(j)}$ does not differ too much from the unperturbed eigenvector $\psi_{\omega,0}^{(j)}$. 

Let $H_{q}^{(l)}[\omega,\varkappa]$ be the operator defined in the proof of Theorem \ref{cl:9}. Let $\widetilde{\psi}_{\omega,\varkappa,l}^{(j)} = D_{\varkappa}\mathcal{L}^{l}\psi_{\omega,\varkappa}^{(j)}$, which satisfies  $H_{q}^{(l)}[\omega,\varkappa]\widetilde{\psi}_{\omega,\varkappa,l}^{(j)}=E_{\omega,\varkappa}^{(j)}\widetilde{\psi}_{\omega,\varkappa,l}^{(j)}$. Let $l=l_{\omega,j}$ and $q_{1}=q_{1}(\varepsilon)$ be such that for $q>q_{1}$, 
\begin{equation}\label{eq:160}
\big|\psi_{\omega,0}^{(j)}(\overline{l}_{q})\big|,\big|\psi_{\omega,0}^{(j)}(\overline{l+q-1}_{q})\big|  \leq  e^{-\big(\gamma\big( E_{\omega,0}^{(j)}\big)-\frac{11}{5}\varepsilon\big)\frac{q}{2}}.
\end{equation}
We have 
\begin{equation}\label{eq:44}
\begin{split}
\big\|\big(H_{q}[\omega,\varkappa]-E_{\omega,0}^{(j)}\big) \psi_{\omega,0}^{(j)}\big\| &=\big\|\big(H_{q}^{(l)}[\omega,\varkappa]-E_{\omega,0}^{(j)}\big) \widetilde{\psi}_{\omega,0,l}^{(j)}\big\|^2 \\
&= \big|e^{iq\varkappa}-1\big|^2\big|\psi_{\omega,0}^{(j)}\big(\overline{l}_{q}\big)\big|^2+\big|e^{-iq\varkappa}-1\big|^2\big| \psi_{\omega,0}^{(j)}\big(\overline{l+q-1}_{q}\big)\big|^2 \\
&< 100e^{-\big(\gamma\big( E_{\omega,0}^{(j)}\big)-\frac{11}{5}\varepsilon\big)q}
\end{split}
\end{equation}
and since the operator $H_{q}[\omega,\varkappa]$ preserves the subspace $ \big(\psi_{\omega,\varkappa}^{(j)}\big)^\perp$, we have 
\begin{equation}\label{eq:43}
\begin{split}
100e^{-\big(\gamma\big( E_{\omega,0}^{(j)}\big)-\frac{11}{5}\varepsilon\big)q}
&>\big\|\big(H_{q}[\omega,\varkappa]- E_{\omega,0}^{(j)}\big) \psi_{\omega,0}^{(j)}\big\|^{2} \geq \big\|\big(H_{q}[\omega,\varkappa]- E_{\omega,0}^{(j)}\big)\varphi_{\omega,\varkappa}^{(j)}\big\|^2 \\
&\geq \big\|\varphi_{\omega,\varkappa}^{(j)}\big\|^2 \min_{k\neq j}\big| E_{\omega,\varkappa}^{(k)}- E_{\omega,0}^{(j)}\big|^2.
\end{split}
\end{equation}

By Theorem \ref{cl:9} there exists $q_{2} = q_{2}(\varepsilon)$ such that if $q>q_{2}$ then for each $j=1,\dots,q$, we have $m\big(B_{\omega,q}^{(j)}\big)<e^{-\big(\gamma\big( b_{\omega,q}^{(j)}\big)-3\varepsilon\big)q}$. There exists $q_{3}=q_{3}(\varepsilon)$ such that for $q>\max\{q_{1},q_{2},q_{3}\}$ and $j=1,\dots,q$,
\begin{equation}\label{eq:113}
 \big\| \psi_{\omega,0}^{(j)}-\alpha_{\omega,\varkappa}^{(j)} \psi_{\omega,\varkappa}^{(j)}\big\|
=\big\|\varphi_{\omega,\varkappa}^{(j)}\big\| 
\leq \frac{ 10e^{-\big(\gamma\big( E_{\omega,0}^{(j)}\big)-\frac{11}{5}\varepsilon\big)\frac{q}{2}}}{e^{-\frac{\varepsilon}{10}q}- e^{-\big(\gamma\big( b_{\omega,q}^{(j)}\big)-3\varepsilon\big)q}}
 \leq e^{-\big(\gamma\big( E_{\omega,0}^{(j)}\big)-\frac{13}{5}\varepsilon\big)\frac{q}{2}}.
\end{equation}

Pythagoras gives $\big|\alpha_{\omega,\varkappa}^{(j)}\big|^{2}+\big\|\varphi_{\omega,\varkappa}^{(j)}\big\|^{2}= 1$ and $\big\| \psi_{\omega,\varkappa}^{(j)}-\psi_{\omega,0}^{(j)}\big\|^{2}=2-2\big|\alpha_{\omega,\varkappa}^{(j)}\big|$, therefore \eqref{eq:113} implies 
$$
\big\|\psi_{\omega,\varkappa}^{(j)}- \psi_{\omega,0}^{(j)}\big\|^{2}<4-4\bigg(1-e^{-\big(\gamma\big( E_{\omega,0}^{(j)}\big)-\frac{13}{5}\varepsilon\big)q}\bigg)^{\frac{1}{2}}.
$$
 There exists $q_{4}=q_{4}(\varepsilon)$ such that 
 $1-\left(1-e^{-\big(\gamma\big(E_{\omega,0}^{(j)}\big)-\frac{13}{5}\varepsilon\big)q}\right)^{\frac{1}{2}}<e^{-\big(\gamma\big( E_{\omega,0}^{(j)}\big)-\frac{14}{5}\varepsilon\big)q}$, whenever $q>q_{4}$. In particular, 
 $\big\|\psi_{\omega,\varkappa}^{(j)}- \psi_{\omega,0}^{(j)}\big\|^{2} < 4e^{-\big(\gamma\big( E_{\omega,0}^{(j)}\big)-\frac{14}{5}\varepsilon\big)q}$ for $q>\max\{q_{1},q_{2},q_{3},q_{4}\}$. The triangle inequality gives
\begin{equation}\label{eq:117}
\big| \psi_{\omega,\varkappa}^{(j)}(x)\big| < 2e^{-\big(\gamma\big( E_{\omega,0}^{(j)}\big)-\frac{14}{5}\varepsilon\big)\frac{q}{2}} + \big| \psi_{\omega,0}^{(j)}(x)\big|. 
\end{equation}
 For any $j= 1,\dots,q$, if $\|x-\nu_{\omega,j,0}\|>Cn$, then \eqref{eq:117} and Lemma \ref{distancedecay} give 
\begin{equation}\label{eq:63}
\big| \psi_{\omega,\varkappa}^{(j)}(x)\big| < 2e^{-\big(\gamma\big( E_{\omega,0}^{(j)}\big)-\frac{14}{5}\varepsilon\big)\frac{q}{2}} +  e^{-\big(\gamma\big( E_{\omega,0}^{(j)}\big)-2\varepsilon\big)\|x-\nu_{\omega,j,0}\|_{q}} \leq 3e^{-\big(\gamma\big( E_{\omega,0}^{(j)}\big)-\frac{14}{5}\varepsilon\big)\|x-\nu_{\omega,j,0}\|_{q}}
\end{equation} whenever $q>\max\{q_{1},q_{2},q_{3},q_{4}\}$. 

Finally, there exists $q_{5}=q_{5}(\varepsilon)$ such that the right hand side of \eqref{eq:63} is bounded above by 
$e^{-\big(\gamma\big( E_{\omega,0}^{(j)}\big)-3\varepsilon\big)\|x-\nu_{\omega,j,0}\|_{q}}$ if  $q>q_{5}$. Take $q_{0}=\max_{k=1,\dots,5}q_{k}$ and $\nu_{\omega,j}=\nu_{\omega,j,0}$.  
\end{proof}

\section{Proposition 4.}\label{sec:4}

Let $H_{\omega}$ be a discrete Schr\"odinger operator with bounded i.i.d potential (I). Large deviations bounds for the norm of random matrix products go back to the work of Le Page \cite{LePage}. 
Our arguments rely on the following large deviation estimate on the entries of the transfer matrix: 
\begin{equation}\label{polynomials}
P_{\omega,[a,b]}(E) = 
\begin{cases} 
      \det(H_{\omega}\upharpoonright[a,b]-E), & a\leq b\\
       1, & a>b
\end{cases}
\end{equation}
 which can be found in the work of Tsay \cite{Tsay}, where large deviation bounds on the norm are used.

\begin{lemma}[\cite{Tsay}, Theorem 2]\label{matrixelements}
Let $H_{\omega}$ be a Schr\"odinger operator with bounded  i.i.d potential \emph{(I)}. For  any $\varepsilon>0$ there exists $c_{1}=c_{1}(\varepsilon)>0$ and $N_{1}=N_{1}({\varepsilon})$ such that for intervals $[a,b]\subset \mathbb{Z}$ with $N_{1}<b-a+1<\infty$ and $E\in K$, 
\begin{equation}\label{eq:118}
\PR\left\{\omega\in \Omega : \left|(b-a+1)^{-1}\log|P_{\omega,[a,b]}(E)| - \gamma(E) \right| < \varepsilon \right\} > 1-e^{-c_{1}(b-a+1)}.
\end{equation}
\end{lemma}

Let $\#[a,b]\leq q$ and $P_{\omega,q,[a,b]}$ be the function defined in \eqref{polynomials} but with $H_{\omega}$ replaced by $H_{\omega,q}$. Let $\widetilde{G}_{\omega,[a,b],E,q} = (H_{\omega,q}\upharpoonright[a,b]-E)^{-1}$.  For any $E\notin\sigma(H_{\omega,q}\upharpoonright[a,b])$ and $(k_{1},k_{2})\in[a,b]^{2}$, Cramer's rule implies
\begin{equation}\label{eq:170}
|\widetilde{G}_{\omega,[a,b],E,q}(k_{1},k_{2})| =
\begin{cases} 
      \frac{|P_{\omega,q,[a,k_{1}-1]}(E)P_{\omega,q,[k_{2}+1,b]}(E)|}{|P_{\omega,q,[a,b]}(E)|}, & k_{1}\leq k_{2} \\
       \frac{|P_{\omega,q,[a,k_{2}-1]}(E)P_{\omega,q,[k_{1}+1,b]}(E)|}{|P_{\omega,q,[a,b]}(E)|}, & k_{2}\leq k_{1}
\end{cases}.
\end{equation} 

Since $\#[a,b]\leq q$, the diagonal of $H_{\omega,q}\upharpoonright [a,b]$ is i.i.d.

\begin{proof}[Proof of Proposition \ref{c0}]
We begin by showing that for any $\varepsilon>0$ there exists $c_{0}=c_{0}(\varepsilon)>0$ and $N_{0}=N_{0}(\varepsilon)$ such that if $n>N_{0}$ then  for any $E\in K$ and $x\in\Z_{q}$, 
\begin{equation}\label{smallprobsing}
\PR\{\omega\in \Omega:x\in\text{Res}_{\omega}(E,\varepsilon,n)\}\leq e^{-c_{0} n}.
\end{equation}
First note that $\widetilde{G}_{\omega,[k-n,k+n],E,q}(k,k\pm n) = G_{\omega,0,B_{n}(\overline{k}_{q}),E,q}(\overline{k}_{q},\overline{k\pm n}_{q})$ so we can use \eqref{eq:170} together with Lemma \ref{matrixelements} to obtain estimates on the entries $(\overline{k}_{q},\overline{k\pm n}_{q})$ of the Green function on the circle.

Indeed, let $Q([a,b],E,\varepsilon)$ denote the event in \eqref{eq:118}. Fix $k\in[0,q-1]$,$E\in K$ and $\varepsilon>0$. Denote 
$Q_{0}(\varepsilon) = Q([k-n,k-1],E,\varepsilon)$, $Q_{1}(\varepsilon)=Q([k+1,k+n],E,\varepsilon)$ and $Q_{2}(\varepsilon)=Q([k-n,k+n],E,\varepsilon)$.
Take $\omega\in \widetilde{Q}(\varepsilon)=Q_{0}(\varepsilon)\cap Q_{1}(\varepsilon)\cap Q_{2}(\varepsilon)$, then 
\begin{equation}\label{eq:38}
|P_{\omega,q,[k-n,k-1]}(E)|,|P_{\omega,q,[k+1,k+n]}(E)| < e^{(\gamma(E)+\varepsilon)n}
\text{ and }
|P_{\omega,q,[k-n,k+n]}(E)| > e^{(\gamma(E)-\varepsilon)(2n+1)}
\end{equation}
and assuming $N_{0}\geq 1$, \eqref{eq:170} implies 
$$
|G_{\omega,0,B_{n}(\overline{k}_{q}),E,q}(\overline{k}_{q},\overline{k\pm n}_{q})|=|\widetilde{G}_{\omega,[k-n,k+n],E,q}(k,k\pm n)| <  e^{-n(\gamma(E)-3\varepsilon)}e^{\varepsilon}<e^{-n(\gamma(E)-4\varepsilon)}.
$$
 Thus $\widetilde{Q}\left(\frac{\varepsilon}{4}\right)\subseteq \{\omega\in \Omega : x\in\text{Res}^{\mathsf{c}}_{\omega}(E,\varepsilon,n)\}$ and  
$\PR\big(\widetilde{Q}\big(\frac{\varepsilon}{4}\big)\big)\leq \PR\{\omega\in \Omega : x\in\text{Res}^{\mathsf{c}}_{\omega}(E,\varepsilon,n)\}$.

Let $c_{1}=c_{1}(\frac{\varepsilon}{4})$ and $N_{1}=N_{1}(\frac{\varepsilon}{4})$ be the constants from Lemma \ref{matrixelements}. Fix $c_{0}=c_{0}(\varepsilon) < c_{1}$, there exists $N_{2}=N_{2}(\varepsilon)$  such that if $n>N_{2}$ then $ 1 -e^{-c_{1}(2n+1)}-2e^{-c_{1}n} >1-e^{-c_{0}n}$ and in particular 
 $\PR(\widetilde{Q}(\frac{\varepsilon}{4}))>1-e^{-c_{0}n}$ whenever $n>N_{0}(\varepsilon)=\max\{N_{1}(\frac{\varepsilon}{4}),N_{2}(\varepsilon)\}$.

For any $ E\in K$, $x\in\Z_{q}$ and $n>N_{0}$, \eqref{smallprobsing} gives 
\begin{equation}\label{eq:51}
\PR\{
\omega\in \Omega : E\in\text{Res}_{\omega}^*(x,\varepsilon,n)\}=\PR\{\omega\in \Omega:x\in\text{Res}_{\omega}(E,\varepsilon,n)\}\leq e^{-c_{0} n}
\end{equation} where 
$\text{Res}_{\omega}^*(x,\varepsilon,n)  = \{ E\in K : x\in \text{Res}_{\omega}(E,\varepsilon,n)\}$, a set which is made up of at most $2n+1$ closed intervals. Indeed, for any $x\in\Z_{q}$ the function $|G_{\omega,0, B_{n}(x),E,q}(x,x\pm \overline{n}_{q})|$ is, as a function of the energy, the absolute value of a rational function where the numerator has degree $2n$ and the denominator $2n+1$. The set is bounded and each singularity is located in the interior of a closed interval. 

We have $\omega\in Q_{\text{NR}}^{\mathsf{c}}(\varepsilon,n,q)$ if and only if there exists $x,y\in\Z_{q}$ and $ E\in K$ with $\|x-y\|_{q}>2n$ such that $x,y\in\text{Res}_{\omega}(E,\varepsilon,n)$ and in particular if and only if $\text{Res}_{\omega}^*(x,\varepsilon,n)\cap \text{Res}_{\omega}^*(y,\varepsilon,n)\neq \varnothing$. Hence
\begin{equation}\label{eq:49}
Q_{\text{NR}}^{\mathsf{c}}(\varepsilon,n,q) = \bigcup_{\substack{x,y\in\Z_{q} ; \\ \|x-y\|_{q}>2n}} \{\omega\in \Omega :  \text{Res}_{\omega}^*(x,\varepsilon,n)\cap  \text{Res}_{\omega}^*(y,\varepsilon,n)\neq \varnothing \}.
\end{equation}
Let $E_{x,m}^{(i)}(\omega)$ denote the $i$-th edge of the $m$-th interval of $\text{Res}_{\omega}^*(x,\varepsilon,n)$. Fix $\|x-y\|_{q}>2n$ and note 
\begin{equation*}
\{\omega\in \Omega: \text{Res}_{\omega}^*(x,\varepsilon,n)\cap \text{Res}_{\omega}^*(y,\varepsilon,n)\neq \varnothing \} = \bigcup_{\substack{i=1,2 ; \\ 1 \leq m \leq 2n+1}} (Q_{x,y}(i,m)\cup Q_{y,x}(i,m))
\end{equation*}
where $Q_{x,y}(i,m) = \{\omega\in \Omega : E_{x,m}^{(i)}(\omega) \in  \text{Res}_{\omega}^*(y,\varepsilon,n)\}$. By stationarity and the union bound,
\begin{equation}\label{intersection}
\PR \{\omega\in \Omega : \text{Res}_{\omega}^*(x,\varepsilon,n)\cap \text{Res}_{\omega}^*(y,\varepsilon,n)\neq \varnothing \} 
\leq \sum_{\substack{i=1,2 ; \\ 1 \leq m \leq 2n+1}} 2\PR(Q_{x,y}(i,m)).
\end{equation}

For $v\in\Omega_{0}^{2n+1}$ let $H[v]$ denote a tridiagonal matrix with $1$'s on the off diagonals and  $v$ on the main diagonal. Let $G_{v,E}=(H[v]-E)^{-1}$ with indices in $[0,2n]^{2}$ and $R(v,\varepsilon,n)=\{ E\in K:|G_{v,E}(n,n\pm n)|>e^{-(\gamma(E)-\varepsilon)n} \}$.  Let $E(v,\varepsilon,n)$ be the $i$-th edge of the $m$-th interval in $R(v,\varepsilon,n)$ and define 
$$Q_{1}=\left\{\begin{pmatrix} v_{1} \\ v_{2}\end{pmatrix}, v_{1},v_{2}\in\Omega^{M}: E(v_{1},\varepsilon,n)\in R(v_{2},\varepsilon,n)\right\}.$$ 

By stationarity, for any $\|x'-y'\|_{q}>2n$, we have $\PR(Q_{x,y}) = \PR(Q_{x',y'})$ so we pick $x=\overline{n}_{q}$ and $y=\overline{3n+1}_{q}$ then indeed $\|x-y\|_{q}>2n$ and $Q_{x,y}(i,m)= \Omega_{0}^{\N}\times Q_{1}\times \Omega_{0}^{\N}$. By consistency of the distribution we have $\PR(Q_{x,y}(i,m))= \mu^{2(2n+1)}(Q_{1})$. Let  $V\sim\mu^{2(2n+1)}$ be a random variable, Fubini and \eqref{eq:51} imply
\begin{equation}\label{eq:48}
\PR(Q_{x,y}(i,m)) = \E_{\mu^{2(2n+1)}}[\1_{Q_{1}}(V)] =\E_{\mu^{2n+1}}[ \E_{\mu^{2n+1}}[\1_{Q_{1}}(V)]]< \E_{\mu^{2n+1}}[ e^{-c_{0}n}] =e^{-c_{0}n}
\end{equation}
 where we first fixed the edge by fixing the first $2n+1$ random variables and then we applied \eqref{smallprobsing} and subsequently integrated over the second $2n+1$ random variables. 

Taking the union bound on \eqref{eq:49} then applying \eqref{intersection} and \eqref{eq:48} gives, for $n>N_{0}$, 
\begin{equation}\label{unionbound}
\PR(Q_{\text{NR}}^{\mathsf{c}}(\varepsilon,n,q))
\leq \sum_{\substack{\|x-y\|_{q}>2n; \\ \|x\|_{q},\|y\|_{q}\leq  \frac{q}{2}}} \sum_{\substack{i=1,2 ; \\ 1 \leq m \leq 2n+1}} 2\PR(Q_{x,y}(i,m)) \leq 12q^{2}ne^{-c_{0}n}.
\end{equation}  Let $c<c_{0}$ and increase $N_{0}$ so that for $n>N_{0}$, we have $12q^{2}ne^{-c_{0}n} < q^{2}e^{-cn}$. 
\end{proof}

\section{Eigenvalue separation.}\label{sec:5}

Let $c_{0}({\varepsilon})$ be the constant from the statement of Proposition \ref{c0}. Then the lower bound on the probability in the conclusion of Proposition \ref{c0} is of use if $n > C'(\varepsilon)\log q$ where $C'(\varepsilon)>\frac{3}{c_{0}(\varepsilon)}$, since we assume $n<\frac{q}{2}$. In Lemma \ref{vectorsep} and Proposition  \ref{eigensep} we fix a constant $\varepsilon_{1}>0$ for which $2\varepsilon_{1}<\min\gamma$ so that we can use Lemma \ref{distancedecay} and such that $c_{0}(\varepsilon_{1})\leq 3$, which is possible since $c_{0}(\varepsilon)\rightarrow 0$. This choice of $\varepsilon_{1}$ implies $C'(\varepsilon_{1})>1$ and therefore we may assume $q<e^{n}$. 

The following lemma was given in \cite{BourgainSpacings} for the case of the Bernoulli shift with Dirichlet boundary conditions and is the first ingredient for the proof of Proposition \ref{eigensep}. 

\begin{lemma}\label{vectorsep}
Let $H_{\omega}$ be a Schr\"odinger operator with bounded  i.i.d potential \emph{(I)}. There exists $A_{1},N_{0},A$ and $\eta,\beta>0$ such that if $N_{0}<n<q$, then for any $\omega\in  Q_{\emph{NR}}(\varepsilon_{1},n,q)$ if there exists eigenvalues $E\neq E'$ of $H_{q}[\omega,0]$ which satisfy $|E-E'|<e^{-\beta n}$ and if the maxima $\nu,\nu'\in\Z_{q}$ of the corresponding eigenvectors satisfy $\|\nu-\nu'\|_{q}<\frac{q}{A_{1}}$, then $\|\nu-\nu'\|_{q} > \eta \log\frac{1}{|E-E'|}$.
\end{lemma}

\begin{proof}
Assume $e^{-\frac{1}{2}n\min\gamma}<\frac{1}{2}$, take $\omega\in  Q_{\text{NR}}(\varepsilon_{1},n,q)$ and let $(E,\psi),(E',\psi')$ be eigen-pairs of $H_{q}[\omega,0]$  with corresponding maxima $\nu,\nu'\in\Z_{q}$ and $E\neq E'$. By Lemma \ref{distancedecay} there exists $C$ such that  for $  \|x-\nu\|_{q}>Cn$, 
\begin{equation}\label{eq:180}
|\psi(x)|\leq e^{-(\min\gamma - 2\varepsilon_{1})\|x-\nu\|_{q}}.
\end{equation}
The same is true of the eigenvector $\psi'$ with $\nu$ replaced by $\nu'$.

Let $M=\left\lceil C_{1}n\right\rceil$ where $C_{1}=C+C_{2}$ and choose $C_{2}\geq 0$ so that $|\psi(x)|,|\psi'(x')|<e^{-10n}$, for $\|x'-\nu'\|_{q},\|x-\nu\|_{q}> M$. Let $\Lambda\subset \Z_{q}$ be the smallest arc containing both $\nu,\nu'\in\Z_{q}$ and $M$ extra points in both directions, namely $\Lambda=\{\nu-\overline{\alpha M}_{q},\nu-\overline{\alpha(M-1)}_{q},\dots,\nu'+\overline{\alpha M}_{q}\}$ for some $\alpha=\pm1$. Denote $\delta=|E-E'|$ and $|\Lambda|=\#\Lambda$. 

The present lemma follows from the following claim. 

\begin{claim}\label{claim10}
There exists $N_{0}$ and $\eta'>0$ such that if $n>N_{0}$ then $\log\frac{1}{\delta} < \eta' |\Lambda|$.
\end{claim}

Let $n>N_{0}$ then Claim \ref{claim10} implies $\|\nu-\nu'\|_{q}>\frac{1}{\eta'}\log\frac{1}{\delta}-(2M+1)$. The lemma follows after we show that $\delta<e^{-\beta n}$ for some $\beta>0$ and $n>N_{0}$ implies $\frac{1}{\eta'}\log\frac{1}{\delta}-(2M+1)>\eta\log\frac{1}{\delta}$. We may assume $2M+1\leq 4C_{1}n$ since $C\geq 2$ thus $\frac{1}{\eta'}\log\frac{1}{\delta}-(2M+1)> \frac{1}{\eta'}\log\frac{1}{\delta}-4C_{1}n$, so it's enough to ensure 
$\frac{1}{\eta'}\log\frac{1}{\delta}-4C_{1}n>\eta\log\frac{1}{\delta}$
which is equivalent to $\delta<e^{-\beta n}$ where $\beta = \frac{2C_{1}}{\eta'}$ and we let $\eta=\frac{1}{2\eta'}$.

To establish Claim \ref{claim10} we need to relate the distance between the two eigenvalues $\delta$ to the size of the arc $\Lambda$. This is achieved via the Wronskian, given by 
$$
W(x)=\psi'(x)\psi(x+1)-\psi(x)\psi'(x+1)
$$
for $x\in\Z_{q}$.

By expanding the Wronskian as a determinant and factoring out a transfer matrix, we  obtain 
$W(x)=(E-E')\psi(x)\psi'(x)+W(x-1)$ for any $x\in\Z_{q}$. The Cauchy-Schwarz inequality gives 
$$
\sum_{x\in\mathbb{Z}_{q}}|W(x)-W(x-1)|=|E-E'|\langle |\psi|,|\psi'| \rangle\leq \delta
$$
where $|\psi|$ denotes the vector obtained by taking the modulus of the components of $\psi$.  Applying the reverse triangle inequality we obtain 
$|W(x+1)|\geq |W(x)|-\delta$ and $|W(x-1)|\geq |W(x)|-\delta$. In particular, for $0\leq k\leq \lfloor\frac{q}{2}\rfloor$ we have $|W(\nu+k)|\geq |W(\nu)|-k\delta$ and $|W(\nu-k)|\geq |W(\nu)|-k\delta$ whence for any $x\in\Z_{q}$, 
\begin{equation}\label{eq:10}
|W(x)|\geq |W(\nu)|-\|x-\nu\|_{q}\delta.
\end{equation}

Localisation \eqref{eq:180} gives 
$|W(x)|\leq |\psi({x+1})|+|\psi(x)|\leq 2e^{- (\min\gamma - 2\varepsilon_{1})(\|x-\nu\|_{q}-1)}$
 and \eqref{eq:10} gives 
\begin{equation}\label{eq:184}
\|x-\nu\|_{q}\delta + 2e^{- (\min\gamma - 2\varepsilon_{1})(\|x-\nu\|_{q}-1)} \geq  |W(\nu)|
\end{equation}
for $\|x-\nu\|_{q}>M+1$. This motivates 
 
\begin{claim}\label{claim8}
There exist $S>1$ and $N_{1}$ such that if $n>N_{1}$, then 
$|W(\nu)|>  S^{-|\Lambda|} - \delta$.
\end{claim}
\begin{proof}[Proof of Claim \ref{claim8}]
Define the function $w(x)=\psi'(\nu)\psi(x)-\psi(\nu)\psi'(x)$ and let $S_{1}=\sup_{(\omega,q)}\|H_{\omega,q}\|$. 
For any $x\in\mathbb{Z}_{q}$ we show 
\begin{equation}\label{eq:181}
|W(\nu)|\geq  |w(x)|( 2+2S_{1})^{-\|x-\nu\|_{q}} - \delta.
\end{equation}

Indeed, we have $\|(H_{q}[\omega,0]-E)w\|\leq \delta$ which implies, for each $x\in\Z_{q}$, 
$$|w(x+1)+(V_{\omega,q}(x)-E)w(x)+w(x-1)|\leq \delta$$ then both triangle inequalities give 
$$
|w(x+1)|\leq \delta+|V_{\omega,q}(x)-E||w(x)|+|w(x-1)|
$$
and strong induction whilst noting that $w(\nu)=0$ and $|E|\leq S_{1}$ (since $E\in\sigma(H_{\omega,q})$), gives 
\begin{equation}\label{eq:5}
|w({\nu+k})|\leq (2+2S_{1})^{|k|}(|w(\nu+1)|+\delta)
\end{equation}
for any $|k| \leq  \lfloor \frac{q}{2}\rfloor$. 

Since $\|x-\nu\|_{q} \leq \lfloor \frac{q}{2}\rfloor$ and either $x = \nu+\overline{\|x-\nu\|_{q}}_{q}$ or $x = \nu-\overline{\|x-\nu\|_{q}}_{q}$, it follows from \eqref{eq:5}, that $|w({x})|\leq ( 2+2S_{1})^{\|x-\nu\|_{q}}(|w(\nu+1)|+\delta).$ \eqref{eq:181} follows after substituting  $w(\nu+1)=W({\nu})$. 

We show that for large enough $n$ we have 
\begin{equation}\label{eq:0}
\max_{x\in\Lambda}|w(x)|>\frac{1}{\sqrt{8M|\Lambda|}}.
\end{equation}

Indeed, since $q<e^{n}$, let $N_{1}\geq (\min\gamma)^{-1}\log(4)$ be such that $(q-(2M+1))e^{-20n}<\frac{1}{2}$ for $n>N_{1}$ then
$$
1 = \sum_{\|x-\nu\|_{q}\leq M} + \sum_{\|x-\nu\|_{q} > M} |\psi(x)|^{2}
< 2M|\psi(\nu)|^2+(q-(2M+1))e^{-20n}
$$
implies $|\psi(\nu)| > \frac{1}{2\sqrt{M}}$ and, since $\psi\perp\psi'$ and $|\psi(x)|,|\psi'(x)|<e^{-10n}$ for any $x\in {\Lambda}^{\mathsf{c}}$, 
$$
\frac{1}{4M}<|\psi(\nu)|^2
< \|w\|^2 \leq |{\Lambda}^{\mathsf{c}}|(2e^{-20n})+|\Lambda|\max_{x\in\Lambda}|w(x)|^2.
$$
Then substituting $|{\Lambda}^{\mathsf{c}}|= q-|\Lambda|$ gives
$$
\max_{x\in\Lambda}|w(x)|^2 
> \frac{1}{|\Lambda|}\left(\frac{1}{4M}-2e^{-20n}|{\Lambda}^{\mathsf{c}}|\right) 
=\frac{1}{4M|\Lambda|}-2e^{-20n}\frac{q-|\Lambda|}{|\Lambda|}
$$
and increase $N_{1}$ such that
$\left|2e^{-20n}\frac{q-|\Lambda|}{|\Lambda|}\right|<4qe^{-20n}<\frac{1}{8M|\Lambda|}$
for $n>N_{1}$.  \eqref{eq:0} follows for $n>N_{1}$.

Assume $n>N_{1}$ and let $x_{0}=x_{0}(n)$ be the maximum of $w$. Applying \eqref{eq:0} to \eqref{eq:181} gives
\begin{equation}\label{eq:183}
|W(\nu)|+\delta > \frac{1}{\sqrt{8M|\Lambda|}}( 2+2S_{1})^{-\|x_0-\nu\|_{q}} > \frac{1}{\sqrt{8M|\Lambda|}}( 2+2S_{1})^{-|\Lambda|}.
\end{equation}

Finally, we show that there exists $S>0$ such that for large enough $n$,  
the RHS of \eqref{eq:183} is greater than $S^{-|\Lambda|}$, which concludes the proof of the claim. It suffices to show 
$\log S>\frac{1}{2}|\Lambda|^{-1}\log(8M|\Lambda|)+\log( 2+2S_{1})$. Indeed, $M < |\Lambda|$ implies that the first summand on the right tends to $0$ as $n\rightarrow\infty$, so increase $N_{1}$ such that $|\Lambda|^{-1}\log(\sqrt{8}|\Lambda|)\leq \log( 2+2S_{1})$ for $n>N_{1}$ and fix $S>(2+2S_{1})^2$.
\end{proof}

Let $S$ and $N_{1}$ be as in Claim \ref{claim8} and assume $n>N_{1}$ then \eqref{eq:184} gives 
\begin{equation}\label{eq:185}
\delta > \frac{S^{-|\Lambda|}- 2e^{- (\min\gamma - 2\varepsilon_{1})(\|x-\nu\|_{q}-1)}}{\|x-\nu\|_{q}+1}.
\end{equation}
We show that there exists $\eta''>0$ and a point $x=x(n)\in\Z_{q}$ such that for large enough $n$, the RHS of \eqref{eq:185} is bounded below by $ \frac{\eta''S^{-|\Lambda|}}{|\Lambda|}$. Indeed, it suffices to show 
\begin{equation}\label{eq:61}
|\Lambda|\left(\frac{1-2e^{- (\min\gamma - 2\varepsilon_{1})(\|x-\nu\|_{q}-1)+|\Lambda|\log S}}{\|x-\nu\|_{q}+1}\right)>\eta''.
\end{equation}
If we pick $x=x(n)$ for which $\|x-\nu\|_{q}>1+\frac{|\Lambda|\log  S + \log4}{\min\gamma - 2\varepsilon_{1}}$
 then   $2e^{- (\min\gamma - 2\varepsilon_{1})(\|x-\nu\|_{q}-1)+|\Lambda|\log   S}<\frac{1}{2}$.  We require \eqref{eq:184} so $x(n)$ must satisfy $\|x-\nu\|_{q}>M+1$ and hence choose
\begin{equation}\label{eq:60}
\|x-\nu\|_{q} = \left\lceil M+1+1+\frac{|\Lambda|\log  S+\log4}{\min\gamma - 2\varepsilon_{1}}\right\rceil.
\end{equation}
  
For such $x\in\Z_{q}$, the LHS of \eqref{eq:61} is at least $\frac{|\Lambda|}{2(\|x-\nu\|_{q}+1)}$. Applying $M\leq|\Lambda|$ one obtains 
\begin{equation}\label{eq:16}
\frac{|\Lambda|}{2(\|x-\nu\|_{q}+1)} 
\geq \frac{|\Lambda|}{2\left(\left\lceil |\Lambda|+2+\frac{|\Lambda|\log  S+\log4}{\min\gamma - 2\varepsilon_{1}}\right\rceil+1\right)}
\end{equation}
the RHS of which is a strictly increasing, bounded function. Therefore there exists a bounded limit $y_{0}>0$ as $n\rightarrow \infty $. Picking $\eta''=\frac{y_{0}}{1.0001}$, then there exists a $N_{2}$ such that if $n>N_{2}$ then the RHS of \eqref{eq:16} is greater than $\eta''$.

Such a point $x\in\Z_{q}$ exists provided the circumference of $\Z_{q}$ is large enough and the centres $\nu,\nu'\in\Z_{q}$ are not too far away from each other. Indeed, let $A_{1}$ be such that $\frac{q\log  S}{ (\min\gamma - 2\varepsilon_{1})A_{1}}<\frac{q}{4}$ and let $A$ be such that, using \eqref{eq:60},
 $$
 \|x(n)-\nu\|_{q}\leq C_{1}\frac{q}{A}+4+\frac{(\frac{q}{A_{1}}+2C_{1}\frac{q}{A})\log  S+\log4}{\min\gamma - 2\varepsilon_{1}}<\frac{q}{4}+\frac{q\log  S}{ (\min\gamma - 2\varepsilon_{1}) A_{1}}<\frac{q}{2}
 $$ therefore such a point $x(n)\in\Z_{q}$ satisfying \eqref{eq:60} exists.  

Finally, rearranging $\delta > \frac{\eta''S^{-|\Lambda|}}{|\Lambda|}$ we obtain 
$\log\frac{1}{\delta}< -\log\eta'' + |\Lambda|\log  S + \log|\Lambda|.$
Letting $\eta'=1+\log S$ then increase $N_{2}$ such that for $n>N_{2}$, the RHS is less than $\eta'|\Lambda|$. This concludes the proof of Claim \ref{claim10} and of Lemma \ref{vectorsep}.
\end{proof}

\begin{lemma}\label{holderexponent}
Let $H_{\omega}$ be a Schr\"odinger operator with bounded  i.i.d potential \emph{(I)}. There exists $D$ and $\alpha>0$ such that for any $E\in \R$ and $M,\rho>0$,
\begin{equation}\label{eq:31}
\PR\bigg\{\omega\in \Omega : \exists \psi \in\C^{M}, \|\psi\|=1 \emph{ st } \|(H_{\omega}\upharpoonright[0,M-1]-E)\psi\|,\max_{x\in\{0,M-1\}}|\psi(x)|<\rho\bigg\}\leq MD\rho^{\alpha}.
\end{equation}
\end{lemma}

\begin{proof}
Fix $ E\in \R$ and $M,\rho>0$.   Partition the integers $\Z$ with intervals $ I_{j}=[a_{j}, b_{j}]$ with $\#  I_{j}= M+2$ and let $\widetilde{I}_{j}=[a_{j}+1,b_{j}-1]$.  Then there are at least $k=\left\lfloor\frac{2L+1}{M+2}\right\rfloor$  mutually disjoint subintervals $ I_{j}\subset [-L,L]$. For each $1 \leq j\leq k$, define the event 
\begin{equation}\label{eq:29}
Q_{j}(\rho)=\big\{\omega\in \Omega: \exists \psi\in\mathbb{C}^{[-L,L]}, \text{supp}(\psi)\subset \widetilde{I}_{j},\|\psi\|=1,\big\|(H_{\omega}\upharpoonright \widetilde{I}_{j}-E)\psi\upharpoonright{\widetilde{I}_{j}}\big\|,\big|\psi\upharpoonright{\partial \widetilde{I}_{j}}\big|<\rho\big\}.
\end{equation}
For any $j$, if $\omega\in Q_{j}(\rho)$, then there exists $\psi^{(j)}_{\omega}\in\mathbb{C}^{[-L,L]}$ satisfying the conditions of \eqref{eq:29} and so
\begin{equation}\label{eq:30}
\|(H_{\omega}\upharpoonright[-N,N]-E)\psi^{(j)}_{\omega} \| \leq \|(H_{\omega}\upharpoonright \widetilde{I}_{j}-E)\psi^{(j)}_{\omega}\upharpoonright{\widetilde{I}_{j}} \| + 2\max_{x\in\partial \widetilde{I}_{j}}\left|\psi^{(j)}_{\omega}(x)\right|< 3\rho.
\end{equation}

Given any $\omega\in \Omega$, define the indexing set   $J_{\omega}=\{j:\omega\in Q_{j}(\rho)\}$. Then for any  $i,j\in J_{\omega}$ with $j\neq i$, there exist orthonormal vectors $\psi^{(j)}_{\omega},\psi^{(i)}_{\omega}\in\C^{[-L,L]}$ which remain orthogonal under the transformation $H_{\omega}\upharpoonright[-L,L]$. Let $\varphi_{\omega}=\sum_{j\in J_{\omega}}c_{\omega,j}\psi^{(j)}_{\omega}$ with $\|\varphi_{\omega}\|=1$ then by orthogonality and \eqref{eq:30},
$$\|(H_{\omega}\upharpoonright[-L,L]-E)\varphi_{\omega}\|^{2}=\sum_{j\in J_{\omega}}|c_{\omega,j}|^{2}\|(H_{\omega}\upharpoonright[-L,L]-E)\psi^{(j)}_{\omega}\|^{2}<(3\rho)^{2}$$ which, is true for any $\varphi_{\omega}\in\text{span}\{\psi^{(j)}_{\omega}:j\in J_{\omega}\}$ with $\|\varphi_{\omega}\|=1$. The minimax theorem implies that the number of eigenvalues (counting multiplicity) in the $3\rho$-neighbourhood of $ E\in K$ is at least the cardinality of the indexing set $J_{\omega}$, namely 
$$f_{E,[-L,L]}(\omega)=\#\sigma(H_{\omega}\upharpoonright[-L,L])\cap[E-3\rho,E+3\rho]\geq \#J_{\omega}=\sum_{j=1}^{k}\mathds{1}_{Q_{j}(\rho)}(\omega).$$

Let $Q(\rho,M)$ denote the event in \eqref{eq:31} then since the interval $\widetilde{I}_{j}$ has a size of $\#\widetilde{I}_{j}=M$, stationarity implies $\PR (Q_{j}(\rho))=\PR (Q(\rho,M))$ for any $j$, therefore 
$$
\E [f_{E,[-L,L]}]\geq \E \sum_{j=1}^{k}\mathds{1}_{Q_{j}(\rho)} = \sum_{j=1}^{k}\PR (Q_{j}(\rho))=k\PR (Q(\rho,M)).
$$
In \cite{LePage2} it was shown that the integrated density of states is H\"older continuous of order $\alpha>0$ so $\mathcal{N}([E-3\rho,E+3\rho])\leq D'\rho^{\alpha}$, for  some $D'$,  therefore there exists $D$ such that 
$$\PR (Q(\rho,M))\leq \lim_{L\rightarrow\infty}\left\lfloor\frac{2L+1}{M+2}\right\rfloor^{-1}\E [f_{E,[-L,L]}]=(M+2)\mathcal{N}([E-3\rho,E+3\rho])\leq MD\rho^{\alpha}.$$ 
\end{proof}

\begin{proof}[Proof of Proposition \ref{eigensep}] Let $\varepsilon_{1}>0$ be the constant defined at the beginning of this section. Let $c_{0},C_{2}>0$ be constants to be defined later  and let $M=\lceil C_{2}\log \frac{1}{\delta_{0}}\rceil$ where $\delta_{0} = e^{-c_{0}n}$. Assume $e^{-\frac{1}{2}n\min\gamma}<\frac{1}{2}$ and fix $\varepsilon>0$. Let $I,I'\subset\Z_{q}$ be arcs with size $\#I,\#I'=2M+1$ and let $E_{i,I}(\omega)$ denote the $i$-th eigenvalue of the restriction $H_{q}[\omega,0]\upharpoonright I$. For any $ E\in K$ define 
$$
Q_{I}(E) = \left\{\omega\in \Omega: \exists \psi\in\mathbb{C}^{I}, \|\psi\|=1 
\text{ st } 
\|(  H_{q}[\omega,0]\upharpoonright I -E)\psi\|<9\delta_{0},  |\psi\upharpoonright{\partial I}|<\delta_{0} \right\}.
$$ 
We show that there exists $M=M(c_{0})$ and $A_{\star} = A_{\star}(\varepsilon,c_{0})$ such that for large enough $n$ and $q>A_{\star}n$, 
\begin{equation}\label{eq:69.5}
Q_{\text{Sep}}^{\mathsf{c}}(\varepsilon, q)\cap  Q_{\text{NR}}(\varepsilon_{1},n,q)\subseteq \bigcup_{\substack{I\cap I'=\varnothing; \\ 1\leq i\leq 2M+1}}Q_{I'}(E_{i,I}).
\end{equation}

Indeed, let $(E,\psi),(E',\psi')$ be eigen-pairs of the operator $H_{q}[\omega,0]$  with centres $\nu,\nu'\in\Z_{q}$ and $E\neq E'$. Let $\Lambda, \Lambda'\subset\Z_{q}$ denote balls of radius $M$ centred about $\nu$ and $\nu'$ respectively. Let $\widetilde\psi=\frac{\psi}{\left\|\psi\upharpoonright{\Lambda}\right\|}$ and $ \widetilde\psi'=\frac{\psi'}{\left\|\psi'\upharpoonright{\Lambda'}\right\|}$. We will show that it is the restriction $\widetilde{\psi}'$ which satisfies the conditions of $Q_{\Lambda'}(E_{i,\Lambda})$ for some eigenvalue $E_{i,\Lambda}$ of the restriction $\widetilde{\psi}$. 

We first ensure $\Lambda\cap\Lambda'=\varnothing$, which occurs when $\|\nu-\nu'\|_{q}>2M$. Let $N_{1}$,$A_{1}$,$\beta$,$\eta,A$ be the constants from Lemma \ref{vectorsep} and define $$A_{\star} = A_{\star}(\varepsilon,c_{0}) = \max\{2(C_{2}c_{0}+1)A_{1},A,2(C_{2}c_{0}+1)\eta^{-1}\varepsilon^{-1},\beta\varepsilon^{-1} ,c_{0}\varepsilon^{-1}\}.$$ 

Assume $|E-E'|<e^{-\varepsilon q}$ and $n>N_{1}$, we show that if $q>A_{\star}n$ then  $\|\nu-\nu'\|_{q}>2M$. Indeed, we can apply Lemma \ref{vectorsep} since $|E-E'|< e^{-\varepsilon q} <e^{-\beta n}$ and $N_{1}<n<\frac{q}{A}$.  If $ \|\nu-\nu'\|_{q}\geq \frac{q}{A_{1}}$, then $ \|\nu-\nu'\|_{q}\geq \frac{q}{A_{1}}>\frac{A_{\star}n}{A_{1}}$ so $A>2(C_{2}c_{0}+1)A_{1}$ implies $\|\nu-\nu'\|_{q}>2M$. On the other hand if $ \|\nu-\nu'\|_{q}<\frac{q}{A_{1}}$, then $\|\nu-\nu'\|_{q}>\eta\log\frac{1}{\delta}$, therefore $\varepsilon q> 2\lceil C_{2}nc_{0}\rceil\eta^{-1}$ implies $\|\nu-\nu'\|_{q}>2M$. 

We now show that the vector $\widetilde{\psi}'$ satisfies the conditions of $Q_{\Lambda'}(E_{i,\Lambda})$. Indeed, assume $N_{1}<n<\frac{q}{A_{\star}}$ and take $\omega\in Q_{\text{Sep}}^{\mathsf{c}}(\varepsilon, q)\cap  Q_{\text{NR}}(\varepsilon_{1},n,q)$ then $\Lambda\cap\Lambda'=\varnothing$ and due to $\varepsilon q>c_{0}n$, we have $|E-E'|<\delta_{0}$. Lemma \ref{distancedecay} gives $C$ such that for $\|x-\nu\|_{q}> Cn$,
\begin{equation}\label{eq:70}
|\psi(x)|<e^{-(\min\gamma-2\varepsilon_{1})\|x-\nu\|_{q}}.
\end{equation}
Let $C_{2}>Cc_{0}^{-1}$ be large enough so that \eqref{eq:70} continues to  hold for $\|x-\nu\|_{q}\geq M$. Letting $C_{2}>((\min\gamma-2\varepsilon_{1})c_{0})^{-1}$, then there exists $N_{2}=N_{2}(c_{0})$ such that for $n>N_{2}$, $q<e^{n}$ implies 
$$
\|\psi\upharpoonright\Lambda\|^{2}=1-\|\psi\upharpoonright\Lambda^{\mathsf{c}}\|^{2}>1-e^{-(\min\gamma-2\varepsilon_{1})C_{2}c_{0}n}(e^{n}+1-2C_{2}c_{0}n)>\frac{1}{4}.
$$
In particular, assume further that $C_{2}>(\min\gamma-2\varepsilon_{1})^{-1}$, then 
\begin{equation}\label{eq:23}
\begin{split}
\|(H_{q}[\omega,0]\upharpoonright\Lambda-E)\widetilde\psi\| 
&=(|\widetilde\psi({\nu-M-1})|^{2}+|\widetilde\psi({\nu+M+1})|^{2})^{\frac{1}{2}}\\
&< 2e^{-(\min\gamma-2\varepsilon_{1})(M+1)}\|\psi\upharpoonright\Lambda\|^{-1}
<4e^{-(\min\gamma-2\varepsilon_{1}) M} < 4\delta_{0}
\end{split}
\end{equation}
and  $|\widetilde{\psi}\upharpoonright{\partial\Lambda}| <e^{-(\min\gamma-2\varepsilon_{1}) M}<\delta_{0}$. In addition, since $C$ is independent from the eigenvectors we have $\|(H_{q}[\omega,0]\upharpoonright\Lambda' -E')\widetilde\psi'\| < 4\delta_{0}$ and 
$|\widetilde{\psi}'\upharpoonright{\partial\Lambda'}| <\delta_{0}.$

It follows from \eqref{eq:23} that $\text{dist}(E,\sigma(H_{q}[\omega,0]\upharpoonright\Lambda ))<4\delta_{0}$ so there exists $1\leq i \leq 2M+1$ such that $|E- E_{i,\Lambda}(\omega)|<4\delta_{0}$ and in particular 
\begin{equation}\label{eq:28}
\begin{split}
\|(H_{q}[\omega,0]\upharpoonright\Lambda'- E_{i,\Lambda}(\omega))\widetilde\psi'\| 
&\leq \|(  H_{q}[\omega,0]\upharpoonright\Lambda' -E')\widetilde\psi'\| + |E'- E_{i,\Lambda}(\omega)|\|\widetilde{\psi}'\| \\
&< 4\delta_{0}+|E'-E|+|E-E_{i,\Lambda}(\omega)| < 9\delta_{0}.
\end{split}	
\end{equation}
It follows from $|\widetilde{\psi}'\upharpoonright{\partial\Lambda'}| <\delta_{0}$ and \eqref{eq:28} that $\omega\in Q_{\Lambda'}(E_{i,\Lambda})$ where $\Lambda\cap\Lambda' = \varnothing$ and in particular $\omega\in \cup_{\substack{I\cap I'=\varnothing; \\ 1\leq i\leq 2M+1}}Q_{I'}(E_{i,I})$, provided that $C_{2}=C_{2}(c_{0})$ is as above, $n> \max\{N_{1},N_{2}\}$ and $ q>A_{\star}n$. 

Let $M$, $q$ and $n$ be as above, then \eqref{eq:69.5} implies
\begin{equation}\label{eq:140}
\PR (Q_{\text{Sep}}^{\mathsf{c}}(\varepsilon,q)\cap  Q_{\text{NR}}(\varepsilon_{1},n,q))\leq \sum_{\substack{I\cap I'=\varnothing; \\ 1\leq i\leq 2M+1}} \PR(Q_{I'}(E_{i,I})).
\end{equation}
By stationarity $\PR(Q_{I'}(E_{i,I})) = \PR(Q_{I}(E_{i,I+2M+1}))$ and since there are $q^{2}$ pairs of arcs on the circle, 
\begin{equation}\label{eq:141}
\sum_{\substack{I\cap I'=\varnothing; \\ 1\leq i\leq 2M+1}} \PR(Q_{I'}(E_{i,I})) =\sum_{\substack{I\cap I'=\varnothing; \\ 1\leq i\leq 2M+1}} \PR(Q_{I}(E_{i,I+2M+1})) \leq (2M+1)q^{2}\max_{1\leq i\leq 2M+1}\PR(Q_{I}(E_{i,I+2M+1})).
\end{equation}

Let $H[v]$ denote a tridiagonal matrix with $1$'s in the off diagonals and $v$ in its diagonal.  Let $E[v]$ be the $i$-th eigenvalue of the matrix $H[v]$. For any $\rho>0$ define
$$Q(\rho)=\left\{\begin{pmatrix} v_{1} \\ v_{2}\end{pmatrix}, v_{1},v_{2}\in\Omega_{0}^{2M+1}:  \exists \psi\in\mathbb{C}^{2M+1}, \|\psi\|=1 
\text{ st } 
\|(H[v_{1}]-E[v_{2}])\psi\|,  \max_{x\in\{0,2M\}}|\psi(x)|<\rho \right\}.$$

Let  $V  \sim\mu^{2(2M+1)}$ be a random variable. It follows from Fubini and Lemma \ref{holderexponent}, for some $\alpha>0$,
\begin{equation}\label{eq:48.5}
\E_{\mu^{2(2M+1)}}[\1_{Q(\rho)}(V)] =\E_{\mu^{2M+1}}[ \E_{\mu^{2M+1}}[\1_{Q(\rho)}(V)]]< \E_{\mu^{2M+1}}[ (2M+1)D\rho^{\alpha}] = (2M+1)D\rho^{\alpha}
\end{equation}
where we fixed the energy by fixing the second $2M+1$ random variables, to which we applied Lemma \ref{holderexponent} and subsequently integrated over the first $2M+1$ random variables. Fix an arc $I\subset \Z_{q}$ with $\#I=2M+1$, then  $Q_{I}(E_{i,I+2M+1})\subseteq \Omega_{0}^{\N}\times Q(9\delta_{0})\times \Omega_{0}^{\N}$ so by consistency of the distribution, 
\begin{equation}\label{eq:142}
\PR(Q_{I}(E_{i,I+2M+1}))\leq \mu^{2(2M+1)}(Q(9\delta_{0}))<(2M+1)D(9\delta_{0})^{\alpha}.
\end{equation}

It follows from \eqref{eq:140},\eqref{eq:141} and \eqref{eq:142}, that 
$$
\PR (  Q_{\text{Sep}}^{\mathsf{c}}(\varepsilon, q)\cap  Q_{\text{NR}}(\varepsilon_{1},n,q)) 
\leq (2M+1)^{2}q^{2}D(9\delta_{0})^{\alpha}
\leq \widetilde{D}q^{2}n^{2}e^{-\alpha c_{0}n}
$$ whenever $n>\max\{N_{1},N_{2}\}$. 
Take $c'(\varepsilon_{1})$, $N_{3}(\varepsilon_{1})$ from Proposition \ref{c0}. Fix $c_{0} = \frac{c'(\varepsilon_{1})}{\alpha}$ and $\widetilde{c}>0$ such that $\widetilde{c}<c'({\varepsilon_{1}})$, then there exists $N_{4}$ such that if $n>N_{4}$ then
$\widetilde{D}q^{2}n^{2}e^{-\alpha c_{0}n}+ q^{2}e^{-c'(\varepsilon_{1})n} < q^{2}e^{-\widetilde{c}n}.$

As well as $n<\frac{q}{A_{\star}}$, we also require $n>\frac{c\varepsilon}{\widetilde{c}}q$ in order to get $e^{-\widetilde{c}n} < e^{-c\varepsilon q}$. Therefore we look for an $n>N_{0}=\max\{N_{1},N_{2},N_{3}(\varepsilon_{1}),N_{4}\}$ which satisfies $\frac{c\varepsilon}{\widetilde{c}}q< n<\frac{q}{A_{\star}}$. Let $\varepsilon_{0}>0$ and $\widetilde{A}>0$ be such that $A_{\star} = \varepsilon^{-1}\widetilde{A}$ for all $\varepsilon\leq\varepsilon_{0}$, then for $\varepsilon<\varepsilon_{0}$ we require  $\frac{c\varepsilon}{\widetilde{c}}q< n<\varepsilon\frac{q}{\widetilde{A}}$. Firstly fix $c<\frac{\widetilde{c}}{\widetilde{A}}$. Let $C_{0} = \max\{\frac{\widetilde{c}}{c}N_{0}, 2(\frac{1}{\widetilde{A}}-\frac{c}{\widetilde{c}})^{-1}\}$ then for $\frac{C_{0}}{q}<\varepsilon<\varepsilon_{0}$ there exists $n$ satisfying  
$N_{0}<\frac{c\varepsilon}{\widetilde{c}}q< n<\frac{\varepsilon}{\widetilde{A}}q$ and 
\begin{equation*}
\begin{split}
 \PR (  Q_{\text{Sep}}^{\mathsf{c}}(\varepsilon, q)) 
 &= \PR (  Q_{\text{Sep}}^{\mathsf{c}}(\varepsilon, q)\cap  Q_{\text{NR}}(\varepsilon_{1},n,q)) +  \PR(  Q_{\text{Sep}}^{\mathsf{c}}(\varepsilon, q)\cap  Q_{\text{NR}}^{\mathsf{c}}(\varepsilon_{1},n,q)) \\
 &< \widetilde{D}q^{2}n^{2}e^{-\alpha c_{0}n}+ q^{2}e^{-c'(\varepsilon_{1})n} < q^{2}e^{-\widetilde{c}n} < q^{2}e^{-c\varepsilon q}.
\end{split}
\end{equation*}
\end{proof}

\end{document}